\documentclass[12pt]{article}

\usepackage{mathtools}
\usepackage{amsmath,diagbox}
\usepackage{indentfirst}
\usepackage{mathrsfs}
\usepackage{amsfonts}
\usepackage{arydshln}
\usepackage{enumerate}
\usepackage{setspace}
\usepackage{amssymb,amsthm,cases}
\usepackage{amsbsy,pifont}
\usepackage{latexsym,diagbox,tabularx}
\usepackage{amsmath,amscd,bbm,multirow}
\usepackage{graphicx,epsfig,extarrows,dsfont,mathtools}
\usepackage[final,allcolors=blue,colorlinks=true]{hyperref}
\usepackage[normalem]{ulem}
\input xy
\usepackage{caption,subcaption} 
\usepackage{tikz}
\usepackage{tikz-cd}
\usetikzlibrary{decorations.pathreplacing,decorations.markings}
\usetikzlibrary{cd}
\usetikzlibrary{shapes.geometric}
\usetikzlibrary{arrows,arrows.meta}
\usetikzlibrary{graphs,positioning}
\usetikzlibrary{matrix,decorations.pathmorphing}
\usetikzlibrary{math,calc,intersections,through,angles,arrows.meta,shapes.geometric,shadows,quotes,spy,datavisualization,datavisualization.formats.functions,plotmarks}
\tikzset{every picture/.style={samples=300,smooth,line join=round,thick,>=stealth}}
\usepackage{enumerate}
\usepackage{fancyhdr}
\fancypagestyle{plain}{%
	\fancyhf{}
	\fancyfoot[C]{\thepage}

}
\usepackage[sectionbib]{chapterbib}

\tikzset{
	on each segment/.style={
		decorate,
		decoration={
			show path construction,
			moveto code={},
			lineto code={
				\path [#1]
				(\tikzinputsegmentfirst) -- (\tikzinputsegmentlast);
			},
			curveto code={
				\path [#1] (\tikzinputsegmentfirst)
				.. controls
				(\tikzinputsegmentsupporta) and (\tikzinputsegmentsupportb)
				..
				(\tikzinputsegmentlast);
			},
			closepath code={
				\path [#1]
				(\tikzinputsegmentfirst) -- (\tikzinputsegmentlast);
			},
		},
	},
	mid arrow/.style={postaction={decorate,decoration={
				markings,
				mark=at position .5 with {\arrow[#1]{stealth}}
	}}},
}

\numberwithin{equation}{section}

\theoremstyle{plain}
\newtheorem{theorem}{Theorem}[section]

\newtheorem{corollary}[theorem]{Corollary}

\newtheorem{lemma}[theorem]{Lemma}

\theoremstyle{definition}
\newtheorem{definition}[theorem]{Definition}

\newtheorem{remark}[theorem]{Remark}

\newtheorem{assumption}[theorem]{Assumption}
\newtheorem{notation}[theorem]{Notation}

\def\XXint#1#2#3{{\setbox0=\hbox{$#1{#2#3}{\int}$}
		\vcenter{\hbox{$#2#3$}}\kern-.5\wd0}}

\DeclareMathSymbol{\subseteq}{\mathrel}{symbols}{"12}
\DeclareMathSymbol{\supseteq}{\mathrel}{symbols}{"13} 
\DeclareMathSymbol{\subsetneq}{\mathrel}{AMSb}{"28}                  \DeclareMathSymbol{\supsetneq}{\mathrel}{AMSb}{"29}    

\DeclareMathSymbol{\nsubseteq}{\mathrel}{AMSb}{"2A}                  

\DeclareMathSymbol{\nsupseteq}{\mathrel}{AMSb}{"2B}

\DeclareMathDelimiter{\langle}{\mathop}{symbols}{"68}{largesymbols}{"0A}
\DeclareMathDelimiter{\rangle}{\mathclose}{symbols}{"69}{largesymbols}{"0B}

\DeclareSymbolFont{txfontsA}{U}{txmia}{m}{it}
\SetSymbolFont{txfontsA}{bold}{U}{txmia}{bx}{it}
\DeclareFontSubstitution{U}{txmia}{m}{it}
\DeclareMathSymbol{\upalpha}{\mathord}{txfontsA}{"0B}
\DeclareMathSymbol{\upbeta}{\mathord}{txfontsA}{"0C}
\DeclareMathSymbol{\upgamma}{\mathord}{txfontsA}{"0D}
\DeclareMathSymbol{\updelta}{\mathord}{txfontsA}{"0E}
\DeclareMathSymbol{\upepsilon}{\mathord}{txfontsA}{"0F}
\DeclareMathSymbol{\upzeta}{\mathord}{txfontsA}{"10}
\DeclareMathSymbol{\upeta}{\mathord}{txfontsA}{"11}
\DeclareMathSymbol{\uptheta}{\mathord}{txfontsA}{"12}
\DeclareMathSymbol{\upiota}{\mathord}{txfontsA}{"13}
\DeclareMathSymbol{\upkappa}{\mathord}{txfontsA}{"14}
\DeclareMathSymbol{\uplambda}{\mathord}{txfontsA}{"15}
\DeclareMathSymbol{\upmu}{\mathord}{txfontsA}{"16}
\DeclareMathSymbol{\upnu}{\mathord}{txfontsA}{"17}
\DeclareMathSymbol{\upxi}{\mathord}{txfontsA}{"18}
\DeclareMathSymbol{\uppi}{\mathord}{txfontsA}{"19}
\DeclareMathSymbol{\uprho}{\mathord}{txfontsA}{"1A}
\DeclareMathSymbol{\upsigma}{\mathord}{txfontsA}{"1B}
\DeclareMathSymbol{\uptau}{\mathord}{txfontsA}{"1C}
\DeclareMathSymbol{\upupsilon}{\mathord}{txfontsA}{"1D}
\DeclareMathSymbol{\upphi}{\mathord}{txfontsA}{"1E}
\DeclareMathSymbol{\upchi}{\mathord}{txfontsA}{"1F}
\DeclareMathSymbol{\uppsi}{\mathord}{txfontsA}{"20}
\DeclareMathSymbol{\upomega}{\mathord}{txfontsA}{"21}
\DeclareMathSymbol{\upvarepsilon}{\mathord}{txfontsA}{"22}
\DeclareMathSymbol{\upvartheta}{\mathord}{txfontsA}{"23}
\DeclareMathSymbol{\upvarpi}{\mathord}{txfontsA}{"24}
\DeclareMathSymbol{\upvarrho}{\mathord}{txfontsA}{"25}
\DeclareMathSymbol{\upvarsigma}{\mathord}{txfontsA}{"26}
\DeclareMathSymbol{\upvarphi}{\mathord}{txfontsA}{"27}                   

\DeclareSymbolFont{ugmL}{OMX}{mdugm}{m}{n}
\SetSymbolFont{ugmL}{bold}{OMX}{mdugm}{b}{n}
\DeclareMathAccent{\wideparen}{\mathord}{ugmL}{"F3}

\def\pt{\partial}

\def\ra{\rightarrow}
\def\bs{\boldsymbol}

\def\s{\subseteq}

\def\e{\epsilon}

\def\ol{\overline}

\def\vp{\varphi}
\def\lg{\langle}
\def\llg{\left\langle}

\def\rg{\rangle}
\def\rrg{\right\rangle}
\def\es{\emptyset}
\def\fa{\forall}

\def\bf{\textbf}
\def\pt{\partial}

\def\om{\omega}
\def\Om{\Omega}
\def\la{\lambda}
\def\al{\alpha}
\def\be{\beta}
\def\de{\delta}
\def\ga{\gamma}

\def\Ga{\Gamma}

\def\ts{\times}

\def\iy{\infty}

\def\f{\frac}

\def\Lra{\Leftrightarrow}

\def\Span{{\rm{span}}}

\def\df{\mathrm d}

\def\wt{\widetilde}
\def\wh{\widehat}

\def\esssup{\operatorname*{ess\ \! sup}}

\def\hra{\hookrightarrow}

\def\loc{{\rm loc}}

\def\diag{{\rm diag}}

\def\mcA{\mathcal{A}}

\def\mcT{\mathcal{T}}

\def\mcM{\mathcal{M}}

	\def\mbT{\mathbb{T}}

	\DeclareMathOperator{\Div}{div}

	\DeclareMathOperator{\dist}{dist}

	\DeclareMathOperator{\supp}{{supp}}

	\newcommand{\R}{\mathbb R}
	\newcommand{\N}{\mathbb N}

\begin{document} 
	

	\title{A Class of Degenerate Hyperbolic Equations with Neumann Boundary Conditions and Its Application to Observability}
	\author{Dong-Hui Yang and Jie Zhong}
	
	\maketitle{}
	\thispagestyle{empty}
	\begin{abstract}
	We establish a mixed observability inequality for a class of degenerate
	hyperbolic equations on the cylindrical domain
	$\Omega=\mathbb{T}\times(0,1)$ with mixed Neumann--Dirichlet boundary
	conditions. The degeneracy acts only in the radial variable, whereas the
	periodic angular variable allows propagation with a strong tangential
	component, making a direct top-boundary observation delicate. For
	$\alpha\in[1,2)$, we prove that the solution can be controlled by a boundary
	observation on the top boundary together with an interior observation on a
	narrow strip. The proof combines a weighted functional framework, improved
	regularity, a cut-off decomposition in the angular variable, a multiplier
	argument for the localized component, and an energy estimate for the
	remainder.
	\end{abstract}
	
	\begin{center}
		\parbox{0.92\textwidth}{\small\noindent
			
			}
	\end{center}
	
	\tableofcontents
	\thispagestyle{empty}
	
	\newpage
	
	\section{Introduction}\label{S1}
	 
	This paper is concerned with the observability of a degenerate hyperbolic
	equation posed on the cylindrical domain $\Omega=\mathbb{T}\times(0,1)$, in
	which the degeneracy acts only in the radial direction while the angular
	variable remains periodic. In the uniformly hyperbolic setting,
	observability and controllability are linked by the Hilbert Uniqueness Method
	(HUM) of J.-L. Lions \cite{Lions}; see also
	\cite{Chen1,Coron,Fattorini1,Fursikov,Lasiecka,Lasiecka3,Rousseau,Russell,Yao,Zuazua}.
	For degenerate hyperbolic equations, however, the geometry of propagation and
	the functional framework become considerably subtler.

	Degenerate hyperbolic equations have been studied in several works
	\cite{Bai1,Cannarsa1,Fragnelli,Gueye,Yang,Zhang}, but the available theory is
	still largely one-dimensional, with \cite{Yang} as a notable higher-dimensional
	exception. In that work, a cut-off method was used to establish
	controllability when the control acts near the degenerate region. The same
	localization philosophy has also appeared in related problems
	\cite{Yang,Yang2}. In particular, \cite{Yang2} used the cut-off method in the
	construction of the Dirichlet map for non-homogeneous degenerate hyperbolic
	equations. These precedents suggest that, in higher-dimensional degenerate
	geometries, localization is not merely a technical convenience but often part
	of the correct analytical structure.
	
	It is natural to ask whether one can observe the system from the nondegenerate
	part of the boundary. In the cylindrical geometry studied here, this question
	becomes substantially more delicate than in the one-dimensional setting.
	In fact, the periodic variable $\theta$ allows propagation with a strong
	angular component, while the degeneracy acts only in the radial direction.
	This mismatch between geometry and observation is the source of the main
	difficulty: a direct top-boundary observation does not isolate the angular
	component in any obvious way.

	The main contribution of this paper is to show that, in the range
	$\alpha\in[1,2)$, the observability problem can still be resolved after a
	suitable localization in the angular variable. The key point is not to
	eliminate the angular propagation, but to separate it analytically. We
	decompose the solution into a component supported away from a narrow angular
	strip and a complementary component concentrated near that strip. The
	localized component is handled by a multiplier argument on a single
	coordinate chart, whereas the remainder is controlled by an energy estimate.
	As a consequence, the final observability inequality naturally takes a mixed
	form: a boundary observation on the top boundary
	$\Gamma=\mathbb{T}\times\{1\}$ together with an interior observation on the
	strip $\omega$.

	From this perspective, the strip $\omega$ is not an auxiliary artifact of the
	proof. It is forced by the geometry of the problem and by the quasimode-type
	obstruction exhibited in Section~\ref{S3}. That mechanism explains why a direct
	top-boundary observation is not the right object to isolate the angular
	component, and why a localized decomposition is the natural replacement.

	The present model therefore provides a concrete higher-dimensional setting in
	which observability can be established for a degenerate hyperbolic equation
	with Neumann boundary conditions. To the best of our knowledge, such results
	do not seem to be available in the existing literature.
	
	A further difficulty comes from regularity. In degenerate hyperbolic
	problems, weak solutions do not automatically possess the level of smoothness
	needed to justify the integrations by parts required by multiplier arguments,
	especially near the degenerate boundary. Thus any observability proof in this
	setting must also address the underlying regularity mechanism.
	
	For degenerate parabolic equations, improved regularity estimates have been
	used successfully in controllability problems; see, for instance,
	\cite{Cannarsa}. The hyperbolic setting is subtler, both because the
	propagation mechanism is different and because the choice of a useful
	multiplier is less rigid. In the present cylindrical geometry, however, the
	regularity of weak solutions can still be improved enough to justify the
	multiplier argument. This extra regularity is an essential part of the proof,
	not a purely auxiliary ingredient. We also expect that the same strategy
	should remain meaningful when $\mathbb{T}$ in Assumption \ref{Assumption (H)}
	is replaced by a smooth manifold without boundary.
	
	We consider the observability of the following degenerate hyperbolic equation:
	\begin{equation}\label{01.21.1}
		\begin{cases}
			\partial_{tt}\vp-\Div(A \nabla \vp)=f, &\text{in }Q, \\
			 \vp=0, &\text{on }   \Ga\ts (0,T), \\ 
			 \f{\pt\vp}{\pt\nu_A}=0, &\mbox{on }\Ga^*\ts (0,T), \\
			 \vp(0)= \vp^0, \pt_t \vp(0)= \vp^1, &\text{in }\Omega,
		\end{cases}
	\end{equation}
	where $A, \Om, Q$ and $\Ga, \Ga^*$,  and $\vp^0, \vp^1$ and $f$ are defined in Assumption \ref{Assumption (H)}. 
	
	\begin{assumption}[Assumptions and Notations]\label{Assumption (H)}
		Let $\N=\{0,1,2,\cdots\}$ and $\N^*=\{1,2,\cdots\}$. 
		
		Let $\Om=\mbT\ts (0,1), \Ga^*=\mbT\ts\{0\}$ and $\Ga=\mbT\ts \{1\}$,  and $T>0$, and $Q=\Om\ts (0,T)$, where $\mbT=\{x\in\R^2\colon |x|=1\}$. Denote  $z=(\theta, r)\in \Om$. 
		
		Let $\al\in [1,2)$ and $w=r^\al$, and $A=\diag(1, w)=\diag(1,r^\al)$.  Denote $\nu_A=A\nu$ on $\pt Q$. 
		
		Let $\vp^0\in H_{\Ga}^1(\Om;w), \vp^1\in L^2(\Om)$ and $f\in L^2(Q)$. Here, the weight Sobolev space $H_\Ga^1(\Om;w)$ will be defined in \eqref{01.14.1}. 
	\end{assumption}
	
\begin{remark}\label{01.14.R-1}
	(1) Equation \eqref{01.21.1} is indeed a degenerate hyperbolic equation with periodic boundary conditions
	\begin{equation*}
		\begin{cases}
			\pt_{tt}y(\theta, r,t)-\Div(A\nabla y(\theta,r,t))=f, &\mbox{in } (0,2\pi)\ts (0,1)\ts (0,T),\\
			y(0, r,t)=y(2\pi, r,t), & \mbox{for all } (r,t)\in (0,1)\ts (0,T),\\
			[r^\al \pt_ry(\theta,r,t)]_{r=0}=0, &\mbox{for all } (\theta,t)\in (0,2\pi)\ts (0,T),\\
			y(\theta,1,t)=0, &\mbox{for all } (\theta,t)\in (0,2\pi)\ts (0,T),\\
			y(0)=y^0, \pt_ty=y^1, &\mbox{in }\Om.
		\end{cases}
	\end{equation*}
	
	(2) We explain the operator $\pt_{\theta\theta}$ in the following:
	
	Let 
	\begin{equation*}
		L_1^2=\Span\left\{f_n^1\equiv \sin n\theta\right\}_{n\in\N^*},\quad L_2^2=\Span\left\{f_n^2\equiv \cos n\theta \right\}_{n\in\N}, \quad L^2(\mbT)=L_1^2\oplus L_2^2. 
	\end{equation*}
	Define 
	\begin{equation*}
		H^k(\mbT)=\left\{f=\sum_{n\in\N^*}a_n^1f_n^1+\sum_{n\in\N}a_n^2f_n^2\colon (a_0^2)^2+\sum_{n\in\N^*} n^{2k}\left((a_n^1)^2+(a_n^2)^2\right)<+\iy\right\}
	\end{equation*}
	with norm
	\begin{equation*}
		\|f\|_{H^k(\mbT)}^2=(a_0^2)^2+\sum_{n\in\N}n^{2k}\left((a_n^1)^2+(a_n^2)^2\right). 
	\end{equation*}
		For each $f= \sum_{n\in\N^*}a_n^1f_n^1+\sum_{n\in\N}a_n^2f_n^2$, we have 
	\begin{equation*}
		\pt_{\theta\theta}: H^k(\mbT)\ra H^{k-2}(\mbT), \mbox{ with } \pt_{\theta\theta}f=-\sum_{n\in\N^*}n^2a_n^1f_n^1-\sum_{n\in\N}n^2a_n^2f_n^2. 
	\end{equation*}
	Hence, we may regard $\partial_\theta$ as an operator acting on $\theta \in [0, 2\pi)$.

		(3) It is clear that the degenerate hyperbolic equation \eqref{01.21.1} is
		degenerate on the boundary $\Ga^*$.

\end{remark}
 
\begin{theorem}\label{02.15.T1}
	Let $\vp^0\in H_\Ga^1(\Om;w)$ and $\vp^1\in L^2(\Om)$. Let $\vp$ be the solution of
		\eqref{01.21.1} with respect to $(\vp^0, \vp^1, f=0)$. If
		$T>\sqrt{2}/(2-\al)$, then there exists a constant $C=C(T,\al,\de_0)>0$
		such that
	\begin{equation*}
		\begin{split}
			\left[(2-\al)T-\sqrt{2}\right]E(0)
			&\leq C\iint_{\Ga\ts (0,T)}(\pt_r\vp)^2\df S\df t\\
			&\quad +C\iint_{\om\ts (0,T)}
			\left[\vp^2+(\pt_t\vp)^2+A\nabla\vp\cdot\nabla\vp\right]\df z\df t,
		\end{split}
	\end{equation*}
	\end{theorem}

\begin{remark}\label{02.15.R1}
The appearance of the interior term on $\omega$ is not an artifact of the proof.
As explained in Section \ref{S3}, the periodic variable $\theta$ allows profiles
with strong angular oscillation and weak radial penetration, so that a purely
top-boundary observation is not the natural quantity to control directly in the
present argument. The cut-off decomposition isolates this angular obstruction on
the strip $\omega$, and the final estimate should therefore be read as a mixed
observability inequality: boundary observation on $\Gamma$ together with
interior observation on $\omega$.
\end{remark}

The proof of Theorem \ref{02.15.T1} rests on three ingredients: a weighted
functional framework adapted to the mixed Neumann--Dirichlet geometry of
\eqref{01.21.1}, a cut-off decomposition $\vp=\psi+\xi$ isolating the angular
obstruction, and the combination of a multiplier estimate for $\psi$ with an
energy estimate for $\xi$.

The paper is organized as follows. In Section \ref{S2} we introduce the
weighted spaces, Green's formula, the Neumann trace, the spectral properties of
the operator $\mcA$, and the well-posedness theory for \eqref{01.21.1}. In
Section \ref{S3} we explain the quasimode-type obstruction that motivates the
cut-off decomposition, derive the multiplier identity for $\psi$, estimate the
localized remainder $\xi$, and complete the proof of Theorem
\ref{02.15.T1}. 
	
\section{Preliminary results}\label{S2}

In this section we collect the analytic ingredients used later in the
observability argument. We introduce the weighted Sobolev spaces associated with
\eqref{01.21.1}, establish Green's formula and the Neumann trace in the present
degenerate setting, describe the spectral properties of the operator $\mcA$,
and recall the well-posedness theory for \eqref{01.21.1}.

\subsection{Solution spaces}
		
We begin with the weighted energy space
\begin{equation*}
	H^1(\Om;w)=\left\{u\in L^2(\Om)\colon \int_\Om \nabla u\cdot A\nabla u\df z<+\iy\right\}.
\end{equation*}
Its inner product and norm are defined by
\begin{equation*}
	(u,v)_{H^1(\Om;w)}=\int_\Om (\nabla u\cdot A\nabla v+uv)\df z, \quad 
	\|u\|_{H^1(\Om;w)}=(u,u)_{H^1(\Om;w)}^\f{1}{2}
\end{equation*}
for all $u,v\in H^1(\Om;w)$. We also introduce
\begin{equation}\label{01.14.1}
	\begin{split} 
	C_\Ga^\iy(\Om)
	&=\{u\in C^\iy(\ol\Om)\colon \mbox{there exists $\de>0$ such that } \dist(\supp u,\Ga)>\de\}, \\
	H_\Ga^1(\Om)
	&=\mbox{the closure of the set $C_\Ga^\iy(\Om)$ in }H^1(\Om),\\
	H_\Ga^1(\Om;w)
	&=\mbox{the closure of the set $C_\Ga^\iy(\Om)$ in }H^1(\Om;w), 
	\end{split} 
\end{equation}
	where $\supp u$ denotes the support of $u$. We then define
	\begin{equation*}
		H_\Ga^{-1}(\Om;w)\mbox{ is the dual space of }H_\Ga^1(\Om;w) \mbox{ with pivot } L^2(\Om).
	\end{equation*}

\begin{remark}\label{01.22.R1}
	It is clear that if $u \in H_\Gamma^1(\Omega)$ (or $u \in H_\Gamma^1(\Omega; w)$), then $u = 0$ on $\Gamma$ in the sense of the classical Sobolev trace.
\end{remark}
	 
Next we set
\begin{equation*}
	H^2(\Om;w)=\left\{u\in H^1(\Om;w)\colon \mcA u\in L^2(\Om)\right\}.
\end{equation*}
Its inner product and norm are defined by
\begin{equation*}
	(u,v)_{H^2(\Om;w)}=\int_\Om (\mcA u)(\mcA v)\df z+(u,v)_{H^1(\Om;w)},\quad \|u\|_{H^2(\Om;w)}=(u,u)_{H^2(\Om;w)}^\f{1}{2}
\end{equation*}
for all $u,v\in H^2(\Om;w)$. Finally, we let
\begin{equation*}
	D(\mcA)=H_\Ga^1(\Om;w)\cap H^2(\Om;w). 
\end{equation*}

It is well known that
\begin{equation*}
	H_\Ga^1(\Om;w), \mbox{ and } H^1(\Om;w), \mbox{  and }  H^2(\Om;w) 
\end{equation*}
are Hilbert spaces; see \cite{GC,Heinonen}.  

\begin{remark}\label{02.11.R1}
	The space $C^\iy(\ol\Om)$ is dense in $H^2(\Om;w)$. 
	
		In fact, $H^2(\Om)$ is contained in $H^2(\Om;w)$ because $\int_\Om \nabla u\cdot A\nabla u\df z\leq \int_\Om |\nabla u|^2\df z$, and 
	\begin{equation*}
		 \int_\Om [\mcA u]^2\df z=\int_\Om \left[\pt_{\theta\theta} u+\al r^{\al-1} \pt_ru +r^\al \pt_{rr}u\right]^2\df z\leq C\|u\|_{H^2(\Om)}^2 
	\end{equation*}
		by $\al\in [1,2)$. On the other hand, since $H^2(\Om;w)\s H_\loc^2(\Om)$, it follows that $H^2(\Om)$ is dense in $H^2(\Om;w)$. Hence $C^\iy(\ol\Om)$ is dense in $H^2(\Om;w)$. 
\end{remark} 

The following lemma is a variant of Hardy's inequality.

\begin{lemma}\label{01.04.L1}
	(1) Let $\al\in (1,2)$. Then there exists a constant $C>0$, depending only on $\al$, such that 
	\begin{equation*}
		\int_\Om r^{\al-2}u^2\df z\leq C\int_\Om r^\al (\pt_ru)^2\df z
	\end{equation*}
	for all $u\in H_\Ga^1(\Om;w)$. 
	
	(2) Let $\al=1$. For each $\be>0,\la>0$, there exists a constant $C>0$, depending only on $\be$ and $\la$, such that 
	\begin{equation*}
		\la\int_\Om r^{-1+\be}u^2\df z\leq C\int_\Om r(\pt_ru)^2\df z
	\end{equation*}
	for all $u\in H_\Ga^1(\Om;w)$. 
\end{lemma}

\begin{proof}
	By density, it suffices to prove the statement for $u\in C_\Ga^\iy(\Om)$. 
	
	(1) We prove the case $\al\in (1,2)$.

	From $u=0$ on $\Ga$, we get
	\begin{equation*}
		\begin{split}
			|u(\theta,r)|^2
			&=\left(\int_r^1\pt_r u(\theta, s)\df s\right)^2\leq \left(\int_r^1 s^{\f{\al+1}{2}}|\pt_ru(\theta,s)|^2\df s\right)\left(\int_r^1 s^{-\f{\al+1}{2}}\df s\right)\\
			&\leq \f{2}{\al-1}r^{\f{-\al+1}{2}}\int_r^1s^\f{\al+1}{2}|\pt_ru(\theta,s)|^2\df s, 
		\end{split}
		\end{equation*}
		and therefore
	\begin{equation*}
		\begin{split}
			\int_0^1 r^{\al-2}|u(\theta, r)|^2\df r
			&\leq \f{2}{\al-1}\int_0^1 \int_r^1r^\f{\al-3}{2}s^\f{\al+1}{2}\left|\pt_r(\theta,s)\right|^2\df s\df r \\
			&=\f{2}{\al-1}\int_0^1\int_0^s r^\f{\al-3}{2}s^\f{\al+1}{2}\left|\pt_ru(\theta,s)\right|^2\df r\df s\\
			&=\f{4}{(\al-1)^2}\int_0^1s^\al |\pt_ru(\theta,s)|^2\df s. 
		\end{split}
	\end{equation*}
	Hence 
	\begin{equation*}
		\begin{split}
			\int_\Om r^{\al-2}u^2\df z=\int_\mbT \int_0^1 r^{\al-2}u^2\df r\df \theta\leq \f{4}{(\al-1)^2}\int_\mbT\int_0^1 r^\al |\pt_ru(\theta,r)|^2\df r\df \theta.
		\end{split}
	\end{equation*}
	
	(2) We prove the case $\al=1$. 
	
	From $\int_\Om r u^2\df z\leq  \int_\Om u^2\df z$, and 
	\begin{equation*}
		\begin{split}
			\int_\Om \left|\nabla (r u^2)\right|\df z&=\int_\Om \left|u^2+2r u\nabla u\right|\df z \leq  5\int_\Om u^2\df z+2\int_\Om \nabla u\cdot A\nabla u\df z, 
		\end{split}
	\end{equation*}
		we obtain $r u^2\in W_0^{1,1}(\Om)$. Then 
	\begin{equation*}
		\begin{split}
			\int_\Om u^2\df z
			&=\int_\Om \pt_{r }\left[r u^2\right]\df z  -2\int_\Om r  u\pt_{r }u\df z=-2\int_\Om r u\pt_{r }u\df z\leq \f{1}{2}\int_\Om u^2\df z+4\int_\Om r \left|\pt_{r }u\right|^2\df z, 
		\end{split}
	\end{equation*}
		This implies that 
	\begin{equation}\label{01.28.1}
		\int_\Om u^2\df z\leq 8\int_\Om r \left|\pt_{r }u\right|^2\df z. 
	\end{equation}
	
	Now, let $\be>0,\la>0$. For each $u\in H_\Ga^1(\Om;w)$, since for each $\de\in (0,\f{1}{8})$ we have  
	\begin{equation*}
		0\leq \int_{\{\de<r <1\}} \left(r ^\f{1}{2}\pt_{r }u+r ^\f{-1+\be}{2}u\right)^2\df z, 
		\end{equation*}
		where $\{\de<r<1\}\equiv \{z=(\theta,r)\in \Om\colon \theta\in\mbT, \de<r<1\}$.
		Thus
	\begin{equation*}
		\begin{split}
			0\leq \int_{\{\de<r <1\}}r (\pt_{r }u)^2\df z+\int_{\{\de<r <1\}} r ^{-1+\be}u^2\df z+2\int_{\{\de<r <1\}}r ^\f{\be}{2}u\pt_{r }u\df z. 
		\end{split}
	\end{equation*}
		This implies that 
	\begin{equation*}
		\begin{split}
			0
			&\leq \int_{\{\de<r <1\}}r (\pt_{r }u)^2\df z+\int_{\{\de<r <1\}} r ^{-1+\be}u^2\df z-\int_{\{r =\de\}}r ^\f{\be}{2}u^2\df z -\f{\be}{2}\int_{\{\de<r <1\}}r ^{-1+\f{\be}{2}}u^2\df z. 
		\end{split}
	\end{equation*}
	Then
	\begin{equation*}
		\begin{split}
			&\int_{\{r =\de\}}r ^\f{\be}{2}u^2\df \theta +\la\int_{\{\de<r <1\}} r ^{-1+\be}u^2\df z\\
			&\leq \int_{\{\de<r <1\}}r (\pt_{r }u)^2\df z+(\la+1)\int_{\{\de<r <1\}} r ^{-1+\be}u^2\df z-\f{\be}{2}\int_{\{\de<r <1\}} r ^{-1+\f{\be}{2}}u^2\df z\\
			&\leq \int_{\{\de<r <1\}}r (\pt_{r }u)^2\df z+\int_{\{\de<r <1\}} \left( (\la+1)r ^\f{\be}{2}-\f{\be}{2}\right)r ^{-1+\f{\be}{2}}u^2\df z.
		\end{split}
	\end{equation*}
		Since $\be>0$, we have
	\begin{equation*}
		\begin{split}
			(\la+1)r ^\f{\be}{2}-\f{\be}{2}\leq 0 \mbox{ for all } 0\leq r \leq \left(\f{\be}{2(\la+1)}\right)^\f{2}{\be}.
		\end{split}
	\end{equation*}
	Denote 
	\begin{equation*}
		\begin{split}
			A=\min\left\{1,\left(\f{\be}{2(\la+1)}\right)^\f{2}{\be} \right\},\quad C(\be,\la)=\sup_{r \in [0,1]}\left\{\left[(\la+1)r ^\f{\be}{2}-\f{\be}{2}\right]r ^{-1+\f{\be}{2}}\right\}. 
		\end{split}
	\end{equation*}
		Since $A\ra 0^+$ as $\la\ra +\iy$ or as $\be\ra 0^+$, and $C(\be,\la)\ra +\iy$ as $\la\ra +\iy$ or as $\be\ra 0^+$ (take $r=(\f{\be}{\la+1})^\f{2}{\be}$), we obtain
	\begin{equation*}
		\begin{split}
			\la\int_{\{\de<r <1\}} r ^{-1+\be}u^2\df z\leq \int_{\{\de<r <1\}}r (\pt_{r }u)^2\df z+C(\be,\la)\int_{\{A<r <1\}}u^2\df z.
		\end{split}
	\end{equation*}
	Letting $\de\ra 0^+$, we get
	\begin{equation*}
		\begin{split}
			\la\int_{\Om} r ^{-1+\be}u^2\df z\leq \int_{\Om}r (\pt_{r }u)^2\df z+C(\be,\la)\int_{\Om}u^2\df z.
		\end{split}
	\end{equation*}
		Together with \eqref{01.28.1}, this yields
	\begin{equation}\label{01.30.2}
		\la \int_{\Om} r ^{-1+\be}u^2\df z\leq C(\be,\la)\int_{\Om}r (\pt_{r }u)^2\df z.
	\end{equation}
	This completes the proof of this lemma. 
\end{proof}

\begin{remark}\label{01.21.R1}
	From Lemma \ref{01.04.L1} we get
	\begin{equation*}
		\begin{split} 
		\int_\Om u^2\df z
		&\leq \int_\Om r^{\al-2}u^2\df z\leq C\int_\Om r^\al (\pt_ru)^2\df z\leq C\int_\Om \nabla u\cdot A\nabla u\df z \mbox{ when }\al\in (1,2), \\
		\int_\Om u^2\df z
		&\leq \int_\Om r^{-1+\f{1}{2}}u^2\df z\leq C\int_\Om r(\pt_ru)^2\df z\leq C\int_\Om \nabla u\cdot A\nabla u\df z \mbox{ when } \al=1
		\end{split} 
	\end{equation*}
	for all $u\in H_\Ga^1(\Om;w)$, where the constants $C>0$ depending only on $\al$. This is the Poincar\'e inequality. Hence, in what follows, we use the inner product and norm 
	\begin{equation*}
		(u,v)_{H_\Ga^1(\Om;w)}=\int_\Om \nabla u\cdot A\nabla v\df z, \quad \|u\|_{H_\Ga^1(\Om;w)}=(u,u)_{H_\Ga^1(\Om;w)}^\f{1}{2}
	\end{equation*}
	on $H_\Ga^1(\Om;w)$, respectively. 
\end{remark}

\begin{lemma}\label{01.21.L1}
	The embedding $H_\Ga^1(\Om;w)\hra L^2(\Om)$ is compact. 
\end{lemma}

\begin{proof}
	Let $\|u_n\|_{H_\Ga^1(\Om;w)}\leq M\ (n\in\N^*)$ for some $M>0$. Then there exists a subsequence of $\{u_n\}_{n\in\N^*}$, still denoted by itself, and $u^*\in H_\Ga^1(\Om;w)$ such that $u_n\ra u^*$ weakly in $H_\Ga^1(\Om;w)$. We will show that there exists a subsequence of $\{u_n\}_{n\in\N^*}$, still denoted by itself, such that $u_n\ra u^*$ strongly in $L^2(\Om)$. Without loss of generality, we assume $u^*=0$ (for otherwise, replacing $u_n$ by $u_n-u^*$ for all $n\in\N^*$). 
	
	(1) Let  $\al\in (1,2)$. Let $\e>0$. Since for all $n\in\N^*$ we have 
	\begin{equation*}
		\begin{split} 
		\int_{\Om_\de}u_n^2\df z
		&\leq \de^{2-\al}\int_{\Om_\de} r^{\al-2}u_n^2\df z\leq \de^{2-\al}\int_\Om \nabla u_n\cdot A\nabla u_n\df z\leq M\de^{2-\al}
		\end{split} 
	\end{equation*}
	by Remark \eqref{01.21.R1}, choosing $\de<\de_0=(\f{1}{4M}\e^2)^\f{1}{2-\al}$, then  $\int_{\Om_{\de_0}}u_n^2\df z<\f{1}{4}\e^2$ for all $n\in\N^*$. Note that $u_n|_{\Om_{\de_0}}\in H^1(\Om_{\de_0})$ for all $n\in\N^*$, and $H^1(\Om_{\de_0})\hra L^2(\Om_{\de_0})$ is compact, then there exists a subsequence, still denoted by itself, such that $u_n|_{\Om_{\de_0}}\ra 0$ strongly in $L^2(\Om_{\de_0})$. Hence there exists $n_0\in\N^*$ such that for all $n\geq n_0$ we have $\|u_n|_{\Om_{\de_0}}\|_{L^2(\Om_{\de_0})}<\f{1}{2}\e$, and this implies that $\int_{\Om}u^2\df z<\e$. 
	
	(2) Let $\al=1$. Taking $\be=\f{1}{2}$ and $\la=1$ in \eqref{01.30.2}, then 
	\begin{equation*}
		\int_\Om r^{-\f{1}{2}}u^2\df z\leq C\int_\Om r(\pt_ru)^2\df z, 
	\end{equation*}
	where the constant $C$ is absolute. By the same argument as (1), we complete the proof of this lemma. 
\end{proof}

\subsection{Green's formula}

Green's formula is the foundation for performing integration by parts for solutions of \eqref{01.21.1}. In this section, we establish this formula. 

We introduce the following spaces
\begin{equation*}
	H(\Div,\Om)=\left\{\bf{u}\in [L^2(\Om)]^2\colon \int_\Om \bf{u}\cdot \bf{u}\df z<+\iy,  \int_\Om [\Div\bf{u}]^2\df z<+\iy\right\}
\end{equation*}
with inner product
\begin{equation*}
	(\bf{u},\bf{v})_{H(\Div,\Om)}=\int_\Om \big(\bf{u}\cdot\bf{v}+[\Div\bf{u}][\Div\bf{v}]\big)\df z; 
\end{equation*}
Additional, we define
\begin{equation*}
	H(\Div,\Om;w^{-1})=\left\{\bf{u}\in [L^2(\Om)]^2\colon \int_\Om \bf{u}\cdot A^{-1}\bf{u}\df z<+\iy, \int_\Om [\Div\bf{u}]^2\df z<+\iy\right\}
\end{equation*}
with inner product 
\begin{equation*}
	(\bf{u},\bf{v})_{H(\Div,\Om;w^{-1})}=\int_\Om\big(\bf{u}\cdot A^{-1}\bf{v}+[\Div\bf{u}][\Div\bf{v}]\big)\df z. 
\end{equation*}

\begin{lemma}\label{01.22.L1}
	The spaces $(H(\Div,\Om),(\cdot,\cdot)_{H(\Div,\Om)})$ and $(H(\Div,\Om;w^{-1}),(\cdot,\cdot)_{H(\Div,\Om;w^{-1})})$ are Hilbert spaces. 
\end{lemma}

\begin{proof}
		It is obvious that $(\cdot,\cdot)_{H(\Div,\Om)}$ and
		$(\cdot,\cdot)_{H(\Div,\Om;w^{-1})}$ are bilinear functionals. Hence
		$H(\Div,\Om)$ and $H(\Div,\Om;w^{-1})$ are inner product spaces. 
	
	Let $\{\bf{u}_n\}_{n\in\N}$ be a Cauchy sequence in $H(\Div,\Om)$. Then there exists $\bf{u}^*$ and $v$ such that 
	\begin{equation*}
		\bf{u}_n\ra \bf{u}^* \mbox{ strongly in } [L^2(\Om)]^2, \mbox{ and } \Div\bf{u}\ra v\mbox{ strongly in } L^2(\Om). 
	\end{equation*}
	Then for each $\psi\in C_0^\iy(\Om)$ we have 
	\begin{equation*}
		\int_\Om v\psi\df z=\lim_{n\ra\iy}\int_\Om [\Div\bf{u}_n]\psi\df z=-\lim_{n\ra\iy}\int_\Om \bf{u}_n\cdot \nabla \psi\df z=-\int_\Om \bf{u}^*\cdot\nabla \psi\df z, 
	\end{equation*}
	i.e., $v=\Div\bf{u}^*$ in the sense of distribution. Hence $v=\Div\bf{u}^*$ in $L^2(\Om)$. Therefore, $H(\Div,\Om)$ is a Hilbert space. 
	
	Let $\{\bf{u}_n\}_{n\in\N}$ be a Cauchy sequence in $H(\Div,\Om;w^{-1})$.  Note that $H(\Div,\Om;w^{-1})\hra H(\Div,\Om)$ is continuous, then there exists $\bf{u}^*$  such that $\bf{u}_n\ra \bf{u}^*\mbox{ strongly in } H(\Div,\Om)$, i.e., 
	\begin{equation*}
		\bf{u}_n\ra \bf{u}^* \mbox{ strongly in } [L^2(\Om)]^2, \mbox{ and } \Div \bf{u}_n\ra \Div\bf{u}^* \mbox{ strongly in }L^2(\Om). 
	\end{equation*}
	Now, since $A^{-\f{1}{2}}\bf{u}_n\ra \bf{h}$ in $[L^2(\Om)]^2$ for some $\bf{h}\in [L^2(\Om)]^2$, then for each $\Xi\in [C_0^\iy(\Om)]^2$ we have
	\begin{equation*}
		\begin{split} 
		\int_\Om \bf{h}\cdot \Xi\df z
		&=\lim_{n\ra\iy}\int_\Om A^{-\f{1}{2}}\bf{u}_n\cdot \Xi\df z=\lim_{n\ra\iy}\int_\Om \bf{u}_n\cdot A^{-\f{1}{2}}\Xi\df z\\
		&=\int_\Om \bf{u}^*\cdot A^{-\f{1}{2}}\Xi\df z=\int_\Om A^{-\f{1}{2}}\bf{u}^*\cdot \Xi\df z
		\end{split} 
	\end{equation*}
	by $A^{-\f{1}{2}}\Xi\in [C_0^\iy(\Om)]^2$. Hence $\bf{h}=A^{-\f{1}{2}}\bf{u}^*$. This proves $H(\Div,\Om;w^{-1})$ is a Hilbert space. 
\end{proof}

\begin{lemma}\label{01.22.L2}
	The space $[C^\iy(\ol\Om)]^2$ is dense in $H(\Div,\Om)$. However, $[C^\iy(\ol\Om)]^2$ is not dense in $H(\Div,\Om;w^{-1})$. 
\end{lemma}

\begin{proof}
	It is obvious that the function $\bf{u}=(1,1)\in [C^\iy(\ol\Om)]^2$, but
	\begin{equation*}
		\|\bf{u}\|_{H(\Div,\Om;w^{-1})}=\int_\Om \left(1+r^{-\al}\right)\df z=+\iy 
	\end{equation*}
	since $\al\in [1,2)$. 
	
	We now prove that the space $[C^\infty(\overline{\Omega})]^2$ is dense in $H(\operatorname{div},\Omega)$. This result is classical and can be found, for instance, in \cite[Theorem 1, Section 2, Chapter IX, Part A]{Dautray} or \cite[Chapter 20]{Tartar}. 
\end{proof}

\begin{definition}\label{01.22.D1}
	We define
	\begin{equation*}
		\ga_\nu:  H(\Div,\Om)\ra H^{-\f{1}{2}}(\Ga^*)
	\end{equation*}
	by 
	\begin{equation*}
		\ga_\nu\bf{u}=\bf{u}\cdot\nu=-u_2 \mbox{ for all } \bf{u}=(u_1,u_2)\in [C^\iy(\ol\Om)]^2, 
	\end{equation*}
	where $H^{-\f{1}{2}}(\Ga^*)$ is the dual space of $H^\f{1}{2}(\Ga^*)$ with pivot space $L^2(\Ga^*)$. 
\end{definition}

\begin{remark}\label{02.09.R1}
	We note that the normal trace operator
	\begin{equation*}
		\gamma_\nu : H(\operatorname{div},\Omega) \to H^{-\f{1}{2}}(\Gamma)
	\end{equation*} 
	defined by
	\begin{equation*}
		\gamma_\nu \mathbf{u} = \mathbf{u} \cdot \nu = u_2 \quad \text{for all } \mathbf{u} = (u_1,u_2) \in [C^\infty(\overline{\Omega})]^2
	\end{equation*}
	is also well defined. Clearly, this case is the same as that in Definition \ref{01.22.D1}.
\end{remark}

\begin{lemma}\label{01.22.L3}
	The operator $\ga_\nu$ is a bounded linear operator. 
\end{lemma}

\begin{proof}
	For each $\bf{u}\in [C^\iy(\ol\Om)]^2$, we have
	\begin{equation*}
		\begin{split}
			\lg \ga_\nu \bf{u}, v\rg_{H^{-\f{1}{2}}(\Ga^*), H^\f{1}{2}(\Ga^*)}
			&=\int_{\Ga^*}(\bf{u}\cdot\nu)v\df S=\int_\Om \bf{u}\cdot \nabla \wt v\df z+\int_\Om \wt v\Div\bf{u}\df z\\
			&\leq 3\|\bf{u}\|_{H^1(\Div,\Om)}\|\wt v\|_{H_\Ga^1(\Om)}
		\end{split}
	\end{equation*}
	for all $v\in H^\f{1}{2}(\Ga^*)$, where  $\wt v\in H_\Ga^1(\Om)$ is an extension of $v$ (i.e., $\wt v=v$ on $\Ga^*$, the definition of $H_\Ga^1(\Om)$ is defined in \eqref{01.14.1}), then 
	\begin{equation*}
		\begin{split} 
		\left|\lg\ga_\nu\bf{u}, v\rg_{H^{-\f{1}{2}}(\Ga^*), H^\f{1}{2}(\Ga^*)}\right| 
		&\leq 3\|\bf{u}\|_{H^1(\Div,\Om)}\inf\{\|\wt v\|_{H_\Ga^1(\Om)}\colon \wt v\mbox{ is an extension of }v\}\\
		&\equiv 3\|\bf{u}\|_{H^1(\Div,\Om)}\|v\|_{H^\f{1}{2}(\Ga^*)}. 
		\end{split} 
	\end{equation*}
	Hence 
	\begin{equation*}
		\|\ga_\nu\bf{u}\|_{H^{-\f{1}{2}}(\Ga^*)}\leq 3\|\bf{u}\|_{H(\Div,\Om)}. 
	\end{equation*} 
	
	Now, from Lemma \ref{01.22.L2}, define
	\begin{equation*}
		\ga_\nu \bf{u}=\lim_{n\ra\iy} \ga_\nu\bf{u}_n
	\end{equation*}
	for  $[C^\iy(\ol\Om)]^2\ni\bf{u}_n\ra \bf{u}$ strongly in $H(\Div,\Om)$, 
	then 
	\begin{equation*}
		\begin{split}
			\|\ga_\nu\bf{u}\|_{H^{-\f{1}{2}}(\Ga^*)}\leq 3\|\bf{u}\|_{H^1(\Div,\Om)}, \mbox{ for all } \bf{u}\in H(\Div,\Om). 
		\end{split}
	\end{equation*}
	Since $\ga_\nu$ is obviously a linear operator, we complete the proof of this lemma. 
\end{proof}

\begin{lemma}\label{01.22.L4}
	For all $\bf{u}\in H(\Div,\Om)$ and all $v\in H_\Ga^1(\Om)$, we have
	\begin{equation*}
		\int_\Om \left([\Div\bf{u}]v+\bf{u}\cdot\nabla v\right)\df z=\lg \ga_\nu \bf{u}, v\rg_{H^{-\f{1}{2}}(\Ga^*), H^\f{1}{2}(\Ga^*)}. 
	\end{equation*}
\end{lemma}

\begin{proof}
	Let $\{\bf{u}_n\}_{n\in\N^*}\s [C^\iy(\ol\Om)]^2$ such that $\bf{u}_n\ra \bf{u}$ strongly in $H(\Div,\Om)$. Then 
	\begin{equation*}
		\begin{split}
			\lg\ga_\nu \bf{u},v\rg_{H^{-\f{1}{2}}(\Ga^*), H^\f{1}{2}(\Ga^*)}
			&=\lim_{n\ra\iy}\lg\ga_\nu\bf{u}_n, v\rg_{H^{-\f{1}{2}}(\Ga^*), H^\f{1}{2}(\Ga^*)}=\lim_{n\ra\iy}\int_{\Ga^*}(\ga_\nu\bf{u}_n) v\df S\\
			&=\lim_{n\ra\iy}\int_\Om \left([\Div\bf{u}_n]v+\bf{u}_n\cdot \nabla v\right)\df z=\int_\Om \left([\Div \bf{u}]v+\bf{u}\cdot\nabla v\right)\df z
		\end{split}
	\end{equation*}
		for each $v\in H_\Ga^1(\Om)$. This completes the proof of the lemma. 
\end{proof}

Now, for each $\de\in [0,\f{1}{8}]$, as in Definition \ref{01.22.D1}, we can define 
\begin{equation*}
	\ga_\nu: H(\Div,\Om)\ra H^{-\f{1}{2}}(\Ga_\de^*). 
\end{equation*}  Then, from Lemma \ref{01.22.L4}, we have
\begin{equation*}
	\begin{split}
		&\lg\ga_\nu \bf{u}, v\rg_{H^{-\f{1}{2}}(\Ga^*), H^\f{1}{2}(\Ga^*)}-\lg \ga_\nu \bf{u}, v\rg_{H^{-\f{1}{2}}(\Ga_\de^*), H^\f{1}{2}(\Ga_\de^*)}=\int_{\Om-\Om_\de}\left([\Div\bf{u}]v+\bf{u}\cdot\nabla v\right)\df z
	\end{split}
\end{equation*}
for all $\bf{u}\in H(\Div,\Om)$ and $v\in H_\Ga^1(\Om)$. Hence
\begin{equation*}
	\lim_{\de\ra 0^+}\lg \ga_\nu \bf{u}, v\rg_{H^{-\f{1}{2}}(\Ga_\de^*), H^\f{1}{2}(\Ga_\de^*)}=\lg\ga_\nu \bf{u}, v\rg_{H^{-\f{1}{2}}(\Ga^*), H^\f{1}{2}(\Ga^*)}
\end{equation*}
for all $\bf{u}\in H(\Div,\Om)$ and $v\in H_\Ga^1(\Om)$.

\begin{lemma}\label{01.22.L5}
	For every $\bf{u}\in H(\Div,\Om)$ and every $v\in H_\Ga^1(\Om)$, the function
	\begin{equation*}
		f(\de;v)=\int_{\Ga_\de^*}v^2\df S,\quad g(\de; \bf{u},v)=\lg \ga_\nu\bf{u}, v\rg_{H^{-\f{1}{2}}(\Ga_\de^*), H^\f{1}{2}(\Ga_\de^*)}
	\end{equation*}
	with respect to $\de\in [0,\f{1}{8}]$ are  continuous. 
\end{lemma}

\begin{proof}
	Let $v\in H_\Ga^1(\Om)$. Since for all $0\leq \de_1\leq \de_2\leq \f{1}{8}$ we have 
	\begin{equation*}
		\begin{split}
			|f(\de_1;v)-f(\de_2;v)|
			&=\left|\int_{\Ga_{\de_1}^*}v^2\df S-\int_{\Ga_{\de_2}^*}v^2\df S\right|=\left|\int_{\Om_{\de_1}-\Om_{\de_2}}\pt_r v^2\df z\right|\\
			&\leq \int_{\Om_{\de_1}-\Om_{\de_2}}\left(v^2+|\nabla v|^2\right)\df z=\|v\|_{H_\Ga^1(\Om_{\de_1}-\Om_{\de_2})}^2, 
		\end{split}
	\end{equation*}
	and then $f(\de;v)$ is a continuous function with respect to $\de\in [0,\f{1}{8}]$. 
	
	Let $\bf{u}\in H(\Div,\Om)$ and $v\in H_\Ga^1(\Om)$. Since for all $0\leq \de_1\leq \de_2\leq \f{1}{8}$ we have 
	\begin{equation*}
		\begin{split}
			|g(\de_1;\bf{u},v)-g(\de_2; \bf{u}, v)|
			&=\left|\lg \ga_\nu\bf{u},\nu\rg_{H^{-\f{1}{2}}(\Ga_{\de_1}^*), H^\f{1}{2}(\Ga_{\de_1}^*)}-\lg \ga_\nu\bf{u},\nu\rg_{H^{-\f{1}{2}}(\Ga_{\de_2}^*), H^\f{1}{2}(\Ga_{\de_2}^*)}\right|\\
			&=\int_{\Om_{\de_1}-\Om_{\de_2}}\left([\Div\bf{u}]v+\bf{u}\cdot\nabla v\right)\df z\\
			&\leq 3\|\bf{u}\|_{H^1(\Div,\Om_{\de_1}-\Om_{\de_2})}\|v\|_{H^1(\Om_{\de_1}-\Om_{\de_2})}, 
		\end{split}
	\end{equation*}
	and then $g(\de;\bf{u},v)$ is a continuous function with respect to $\de\in [0,\f{1}{8}]$. 
\end{proof}

Next we establish the Green formula for $\bf{u}\in H(\Div,\Om;w^{-1})$ and $v\in H_\Ga^1(\Om;w)$, namely,
\begin{equation*}
	\int_\Om \left([\Div\bf{u}]v+\bf{u}\cdot\nabla v\right)\df z=0. 
\end{equation*}
For this, we define
\begin{equation}\label{01.22.1}
	\mcT_w: H(\Div,\Om;w^{-1})\ts H_\Ga^1(\Om;w)\ra\R, \quad \mcT_w(\bf{u},v)= \int_\Om \left([\Div\bf{u}]v+\bf{u}\cdot\nabla v\right)\df z.
\end{equation} 

\begin{lemma}\label{01.22.L6}
	The operator $\mcT_w$ is a bounded bilinear functional. 
\end{lemma}

\begin{proof}
		Clearly, $\mcT_w$ is a bilinear functional. 
	
	Now, for each $\bf{u}\in H(\Div,\Om;w^{-1})$ and every $v\in H_\Ga^1(\Om;w)$, we have
	\begin{equation*}
		\begin{split}
			\left|\mcT_w(\bf{u},v)\right|
			&=\left|\int_\Om \left([\Div\bf{u}]v+\bf{u}\cdot\nabla v\right)\df z\right|\\
			&\leq \left(\int_\Om [\Div\bf{u}]^2\df z\right)^\f{1}{2}\left(\int_\Om v^2\df z\right)^\f{1}{2}+\left(\int_\Om \bf{u}\cdot A^{-1}\bf{u}\df z\right)^\f{1}{2}\left(\int_\Om\nabla v\cdot A\nabla v\df z\right)^\f{1}{2}\\
			&\leq 3\|\bf{u}\|_{H(\Div,\Om;w^{-1})}\|v\|_{H_\Ga^1(\Om;w)}, 
		\end{split}
	\end{equation*}
	hence, $\mcT_w$ is a bounded bilinear functional. This completes the proof of this lemma. 
\end{proof}

\begin{theorem}\label{01.22.T1}
	For each $\bf{u}\in H(\Div,\Om;w^{-1})$, we have
	\begin{equation*}
		\mcT_w(\bf{u},v)=0, \mbox{ for all } v\in H_\Ga^1(\Om;w). 
	\end{equation*}
\end{theorem}

\begin{proof}
	Let $\bf{u}\in [C^\iy(\ol\Om)]^2$. Then for all $0<\de_1\leq \de_2\leq \f{1}{8}$, 
	\begin{equation*}
		\begin{split}
			\int_{\de_1}^{\de_2} \f{1}{\de}g(\de; \bf{u},v)^2\df \de
			&=\int_{\de_1}^{\de_2}\f{1}{\de}\lg \bf{u},v\rg_{H^{-\f{1}{2}}(\Ga_\de^*), H^\f{1}{2}(\Ga_\de^*)}^2\df \de =\int_{\de_1}^{\de_2} \f{1}{\de}\left(\int_{\Ga_\de^*}(\bf{u}\cdot\nu)v\df S\right)^2\df \de\\
			&\leq 2\pi\int_{\de_1}^{\de_2}\int_{\Ga_\de^*}\f{1}{\de}u_2^2v^2\df S\df \de \leq 2\pi\|v\|_{L^\iy(\Om)}^2\int_{\mbT\ts (\de_1,\de_2)}r^{-1}u_2^2\df z
		\end{split}
	\end{equation*}
	for all $v\in C_\Ga^\iy(\ol\Om)$. From Lemma \ref{01.22.L5} we get
	\begin{equation*}
		\int_{\de_1}^{\de_2}\f{1}{\de}g(\de;\bf{u},v)^2\df \de\leq 2\pi\|v\|_{L^\iy(\Om)}^2\int_{\mbT\ts (\de_1,\de_2)}r^{-1}u_2^2\df z\leq 2\pi\|v\|_{L^\iy(\Om)}^2\int_{\mbT\ts(\de_1,\de_2)}\bf{u}\cdot A^{-1}\bf{u}\df z
	\end{equation*}
	for all $\bf{u}\in H(\Div,\Om)$ and all $v\in C_\Ga^\iy(\ol\Om)$, since $r^{-1}=r^{\al-1}r^{-\al}\leq r^{-\al}$. Hence, for all $\bf{u}\in H(\Div,\Om;w^{-1})$ and all $v\in C_\Ga^\iy(\ol\Om)$, 
	\begin{equation*}
		\int_{\de_1}^{\de_2}\f{1}{\de}g(\de;\bf{u},v)^2\df \de\leq 2\pi\|v\|_{L^\iy(\Om)}^2\|\bf{u}\|_{H(\Div,\Om;w^{-1})}<+\iy, 
	\end{equation*}
	this implies that 
	\begin{equation}\label{02.11.7}
		g(0;\bf{u},v)=0
	\end{equation}
	for all $\bf{u}\in H(\Div,\Om;w^{-1})$ and all $v\in C_\Ga^\iy(\ol\Om)$. Hence, from Lemma \ref{01.22.L4}, for any $\bf{u}\in H(\Div,\Om;w^{-1})$ and $v\in C_\Ga^\iy(\ol\Om)$, 
	\begin{equation*}
		\mcT_w(\bf{u},v)=0.
	\end{equation*}
	Finally, from Lemma \ref{01.22.L6}, we get
	\begin{equation*}
		\mcT_w(\bf{u},v)=0, \mbox{ for all } \bf{u}\in H(\Div,\Om;w^{-1}), v\in H_\Ga^1(\Om;w). 
	\end{equation*}
	This completes the proof of this theorem. 
\end{proof}

\begin{corollary}\label{03.16.C1}
	For each $\bf{u}=(u_1,u_2)\in H(\Div,\Om;w^{-1})$, we have
	\begin{equation*}
		\mcT_w(\bf{u},v)=\int_{\Ga} v  u_2\df s, \mbox{ for all } v\in H^1(\Om;w). 
	\end{equation*}
\end{corollary}

\begin{proof}
	Choosing $\zeta=\zeta(r)\in C^\iy(\Om), 0\leq \zeta \leq 1$ such that
	\begin{equation*}
		\zeta=1 \mbox{ on }\left(-\iy, \f{1}{2}\right), \quad \zeta=0 \mbox{ on }\left(\f{3}{4}, +\iy\right). 
	\end{equation*}
	Then $v^1=\zeta v\in H_\Ga^1(\Om;w)$ and $v^2=(1-\zeta)v\in H_0^1(\mbT\ts (\f{1}{4},1))$ and $v=v^1+v^2$. Then 
	\begin{equation*}
		\begin{split}
			\mcT_w(\bf{u},v)
			&=\mcT_w(\bf{u},v^1+v^2)=\int_\Om \left([\Div \bf{u}]v^1+\bf{u}\cdot \nabla v^1\right)\df z+\int_\Om \left([\Div \bf{u}]v^2+\bf{u}\cdot \nabla v^2\right)\df z\\
			&=\int_\Om \left([\Div \bf{u}]v^2+\bf{u}\cdot \nabla v^2\right)\df z=\int_\Om \Div(\bf{u}v^2)\df z=\int_{\Ga}v^2 \bf{u}\cdot \nu\df S=\int_{\Ga}v u_2\df S. 
		\end{split}
	\end{equation*}
		by Remark \ref{02.09.R1} and Theorem \ref{01.22.T1}. This completes the proof of the corollary. 
\end{proof}

\begin{lemma}\label{02.11.L1}
	The space $C^\iy(\ol\Om)\cap H(\Div,\Om;w^{-1})$ is dense in $H(\Div,\Om;w^{-1})$. 
\end{lemma}

\begin{proof} 
	It is easily verified that $[C_0^\iy(\Om)]^2\s H(\Div,\Om;w^{-1})$.   
	
	Let $\bf{h}\in H(\Div,\Om;w^{-1})$ be orthogonal to $[C^\iy(\ol\Om)]^2\cap H(\Div,\Om;w^{-1})$. i.e., 
	\begin{equation*}
		\int_\Om \bf{h} \cdot A^{-1}\bf{v}\df x+\int_\Om\left[\Div \bf{h}\right][\Div \bf{v}]\df x=0, \ \fa \bf{v}\in [C^\iy(\ol\Om)]^2\cap H(\Div,\Om;w^{-1}).
	\end{equation*} 
	Set $g=\Div \bf{h}$, then 
	\begin{equation}\label{02.02.02}
		\int_\Om g\Div \bf{v}\df z=-\int_\Om \bf{h}\cdot A^{-1}\bf{v}\df z, \ \fa \bf{v}\in [C^\iy(\ol\Om)]^2\cap H(\Div,\Om;w^{-1}). 
	\end{equation}
		Hence $\nabla g=A^{-1}\bf{h}$ in the sense of distributions on $\Om$ by $[C_0^\iy(\Om)]^2\s H(\Div,\Om;w^{-1})$. Therefore $A\nabla g=\bf{h}$, and $g\in H^1(\Om;w)$. Define
	\begin{equation*}
		\begin{split} 
			\wt g(z)
			&=
			\begin{cases}
				g(z), &z\in \Om,\\
				0, & z\in \mbT\ts (1,2), 
			\end{cases}\quad 
			\wt {\nabla  g}(z)
			=
			\begin{cases}
				A^{-1}(z)\bf{h}(z), &z\in \Om,\\
				0, &z\in \mbT\ts (1,2), 
			\end{cases}
		\end{split} 
	\end{equation*}
	then, for each $\psi\in [C_0^\iy(\mbT\ts (0,2))]^2$, we have $\psi\in H(\Div,\Om;w^{-1})$ by  $(\supp \psi)\cap \Ga^*=\es$, and hence
	\begin{equation*}
		\begin{split}
			\int_{\mbT\ts (0,2)} \wt{\nabla g}\cdot \psi \df z
			&=\int_\Om \bf{h} \cdot A^{-1}\psi\df z=-\int_\Om g \Div \psi\df z\\
			&=-\int_{\mbT\ts (0,2)}\wt g\Div\psi\df z=\int_{\mbT\ts (0,2)}\nabla \wt g\cdot \psi\df z, 
		\end{split}
	\end{equation*}
	where we used  \eqref{02.02.02} in the second equality. Hence $\wt{\nabla g}=\nabla \wt g$ on $\mbT\ts (0,2)$  and  $\wt g\in H^1(\mbT\ts (0,2);w^{-1})$. This shows that $\wt g|_{\Ga}=0$ by Remark \ref{01.22.R1}, and hence $g\in H_\Ga^1(\Om;w)$.

	Finally, for each $\wt{\bf{h}}\in H(\Div,\Om;w^{-1})$, we have 
	\begin{equation*}
		\begin{split}
			(\bf{h},\wt {\bf{h}})_{H(\Div,\Om;w^{-1})}
			&=\int_\Om \left(\bf{h}\cdot A^{-1}\wt{\bf{h}}+(\Div \bf{h})(\Div \wt {\bf{h}})\right)\df z\\
			&=\int_\Om \left(\wt {\bf{h}}\cdot \nabla g+(\Div \wt {\bf{h}})g\right)\df z=\mcT_w(\wt {\bf{h}}, g)=0
		\end{split}
	\end{equation*}
		by $g\in H_\Ga^1(\Om;w)$ and Theorem \ref{01.22.T1}. Hence $\bf{h}=0$. This implies that $C^\iy(\ol\Om)\cap H(\Div,\Om;w^{-1})$ is dense in $H(\Div, \Om;w^{-1})$. This completes the proof of the lemma. 
\end{proof}

\subsection{Neumann boundary condition}

The following theorem shows that the boundary condition $\f{\pt\vp}{\pt\nu_A}=0$ on $\Ga^*\ts(0,T)$ arises naturally in the equation \eqref{01.21.1}. This differs from the classical Neumann boundary condition for uniformly hyperbolic equations.

\begin{theorem}\label{01.22.T2}
	The following assertions hold true:
	
	(i) The following inclusion and equality hold
	\begin{equation*}
		H(\Div,\Om;w^{-1})\s \ker\ga_\nu=\left\{\bf{u}\in H(\Div,\Om)\colon  \ga_\nu\bf{u}=0\right\}. 
	\end{equation*}
	
	(ii) Suppose $u\in H^2(\Om;w)$. Then $\f{\pt u}{\pt\nu_A}=\ga_\nu(A\nabla u)=0$ on $\Ga^*$. Furthermore, for all $u\in H^2(\Om;w)$ and $v\in H_\Ga^1(\Om;w)$, we have
	\begin{equation*}
		-\int_\Om [\Div(A\nabla u)]v\df z=\int_\Om \nabla v\cdot A\nabla u\df z. 
	\end{equation*}
\end{theorem}

\begin{proof}
	(i) Let $\bf{u}\in H(\Div,\Om;w^{-1})$. For all $v\in C_\Ga^\iy(\ol\Om)$, from Lemma \ref{01.22.L4} and Theorem \ref{01.22.T1},  we have
	\begin{equation*}
		\lg\ga_\nu \bf{u}, v\rg_{H^{-\f{1}{2}}(\Ga^*), H^\f{1}{2}(\Ga^*)}=\mcT_w(\bf{u},v)=0,
	\end{equation*}
	and then $\ga_\nu\bf{u}=0$ according to $C^\iy(\Ga^*)\cong C^\iy(\mbT)=C_0^\iy(\mbT)$ is dense in $H^\f{1}{2}(\Ga^*)=H^\f{1}{2}(\mbT)$. 
	
	(ii) Since $u\in H^2(\Om;w)$, we get $A\nabla u\in H(\Div,\Om;w^{-1})$. From (i) we obtain $\f{\pt u}{\pt \nu_A}=\ga_\nu(A\nabla u)=0$ on $\Ga^*$. 
	
	Finally, from $u\in H^2(\Om;w)$ and $v\in H_\Ga^1(\Om;w)$, we have $A\nabla u\in H(\Div,\Om;w^{-1})$, and 
	\begin{equation*}
		-\int_\Om [\Div(A\nabla u)]v\df z=\int_\Om \nabla v\cdot A\nabla u\df z
	\end{equation*}
	by Theorem \ref{01.22.T1}. This completes the proof of this theorem. 
\end{proof}

\begin{corollary}\label{03.01.C6}
	Let $u\in H^2(\Om;w)$ and $v\in H^1(\Om;w)$. Then 
	\begin{equation*}
		-\int_\Om [\Div(A\nabla u)]v\df z=\int_\Om \nabla v\cdot A\nabla u\df z+\int_{\Ga}\f{\pt u}{\pt\nu_A}v\df S, 
	\end{equation*}
	where $\f{\pt u}{\pt \nu_A}=A\nabla u\cdot \nu$ on $\Ga$. 
\end{corollary}

\begin{proof}
	Choosing $\zeta=\zeta(r)\in C^\iy(\Om), 0\leq \zeta \leq 1$ such that
	\begin{equation*}
		\zeta=1 \mbox{ on }\left(-\iy, \f{1}{2}\right), \quad \zeta=0 \mbox{ on }\left(\f{3}{4}, +\iy\right). 
	\end{equation*}
	Then $v^1=\zeta v\in H_\Ga^1(\Om;w)$ and $v^2=(1-\zeta)v\in H_0^1(\mbT\ts (\f{1}{4},1))$ and $v=v^1+v^2$. Hence
	\begin{equation*}
		\begin{split}
			-\int_\Om [\Div(A\nabla u)]v\df z
			&=-\int_\Om [\Div(A\nabla u)]v^1\df z-\int_{\mbT\ts (\f{1}{4},1)}[\Div(A\nabla u)]v^2\df z\\
			&=\int_\Om \nabla u\cdot A\nabla v^1\df z+\int_{\mbT\ts (\f{1}{4},1)}A\nabla u\cdot \nabla v^2\df z-\int_{\Ga}\f{\pt u}{\pt \nu_A} v^2\df S\\
			&=\int_\Om \nabla u\cdot A\nabla v\df z-\int_{\Ga} \f{\pt u}{\pt \nu_A}v\df S. 
		\end{split}
	\end{equation*}
		This completes the proof of the corollary. 
\end{proof}

\subsection{The spectrum for the partial differential operator $\mcA$}

After defining the Neumann boundary condition, we introduce the partial differential operator $\mathcal A$. Associated with this operator are both degenerate elliptic and degenerate hyperbolic equations. In this subsection, we focus on the existence of weak solutions to the degenerate elliptic equation and on the spectral properties of the operator $\mathcal A$.
More precisely,
\[
	\mcA u=-\Div(A\nabla u)=-\pt_{\theta\theta}u-\pt_r(r^\al \pt_r u).
\]

We consider the following degenerate elliptic equation 
\begin{equation}\label{01.22.2}
	\begin{cases}
		\mcA u=g, &\mbox{in }\Om, \\
		u=0, &\mbox{on }\Ga,\\
		\f{\pt u}{\pt \nu_A}=0, &\mbox{on }\Ga^*, 
	\end{cases}
\end{equation}
where $\Om,\Ga,\Ga^*$  are defined in Assumption \ref{Assumption (H)}, and $g\in L^2(\Om)$. 

We call $u\in H_\Ga^1(\Om;w)\cap H^2(\Om;w)$ a weak solution of \eqref{01.22.2} if
\begin{equation*}
	\int_\Om \nabla u\cdot A\nabla v\df z=\int_\Om gv\df z, \mbox{ for all } v\in H_\Ga^1(\Om;w).
\end{equation*}

\begin{lemma}\label{01.22.L7}
	The equation \eqref{01.22.2} has a unique weak solution $u\in H_\Ga^1(\Om;w)\cap H^2(\Om;w)$.  Moreover, there exists a positive constant $C$, depending only on $\al$, such that 
	\begin{equation}\label{02.11.2}
		\|u\|_{H_\Ga^1(\Om;w)}\leq C\|g\|_{L^2(\Om)}. 
	\end{equation}
\end{lemma}

\begin{proof}
	Define
	\begin{equation*}
		B[u,v]=\int_\Om \nabla u\cdot A\nabla v\df z, 
	\end{equation*}
	then, from Remark \ref{01.21.R1},  for all $u,v\in H_\Ga^1(\Om;w)$, we have 
	\begin{equation*}
		\|u\|_{H_\Ga^1(\Om;w)}^2=B[u,u], \quad |B[u,v]|\leq \|u\|_{H_\Ga^1(\Om;w)}\|v\|_{H_\Ga^1(\Om;w)}.
	\end{equation*}
	Note that $g\in L^2(\Om)\s H^{-1}(\Om;w)$ since 
	\begin{equation*}
		\int_\Om gv\df z\leq \|g\|_{L^2(\Om)}\|v\|_{L^2(\Om)}\leq C\|g\|_{L^2(\Om)}\|v\|_{H_\Ga^1(\Om;w)}
	\end{equation*}
	by Remark \ref{01.21.R1}, hence, from the  Lax-Milgram theorem, there exists a weak unique $u\in H_\Ga^1(\Om;w)$ to the equation \eqref{01.22.2}. Moreover, from $\|u\|_{H_\Ga^1(\Om;w)}^2\leq C\|g\|_{L^2(\Om)}\|u\|_{H_\Ga^1(\Om;w)}$ we get \eqref{02.11.2}. 
	
	Finally, from $\mcA u=g$ in the sense of distributions, we get $u\in H^2(\Om;w)$. We complete the proof of this lemma. 
\end{proof}

The following theorem provides an improved regularity result for solutions of \eqref{01.22.2}. It will be used later in the analysis of the degenerate hyperbolic equation \eqref{01.21.1}. 

\begin{theorem}\label{02.01.T1}
	Let $g\in L^2(\Om)$. Let $u$ be the solution of \eqref{01.22.2} with respect to $g$. Then
	\begin{equation*}
		\pt_{\theta\theta}u\in L^2(\Om), \mbox{ and } r^\f{\al}{2}\pt_{\theta r}u\in L^2(\Om),\mbox{ and } r^{1+\f{\al}{2}}\pt_{rr}u\in L^2(\Om). 
	\end{equation*}
	Moreover, there exists a positive constant $C$, depending only on $\al$, such that 
	\begin{equation}\label{02.11.5}
		\int_\Om \left[(\pt_{\theta\theta}u)^2+r^\al (\pt_{\theta r}u)^2+r^{2+\al}(\pt_{rr}u)^2\right]\df z\leq C\int_\Om g^2\df z. 
	\end{equation}
\end{theorem}

\begin{proof}
	We denote the difference quotient
	\begin{equation*}
		D_\theta^hu=\f{1}{h}\big(u(\theta+h,r)-u(\theta,r)\big). 
	\end{equation*}
	For each $\theta_0\in \mbT$, choosing $\de_0\in (0, \f{1}{16})$ and  $\zeta=\zeta(\theta)\in C^\iy(\mbT), 0\leq \zeta\leq 1$ such that 
	\begin{equation*}
		\zeta=0 \mbox{ on } (\theta_0-\de_0,\theta_0+\de_0), \quad \zeta=1 \mbox{ on }\mbT-(\theta_0-2\de_0,\theta_0+2\de_0),  
	\end{equation*}
	and $|\zeta'|\leq C \mbox{ and } |\zeta''|\leq C$, 
	where the constants $C>0$ are absolute. 
	Taking 
	\begin{equation*}
		v=-D_\theta^{-h}\left(\zeta^2D_\theta^h u\right) \mbox{ for } h\in (0,\de_0),
	\end{equation*}
	then $v\in H_\Ga^1(\Om;w)$. Using $v$ as the test function, from \eqref{01.22.2}, we get
	\begin{equation*}
		\begin{split}
			A_1\equiv \int_\Om \nabla u\cdot A\nabla v\df z=\int_\Om gv\df z\equiv A_2. 
		\end{split}
	\end{equation*}
	
	Note that for the functions $u,v$ on $\Om$, we have 
	\begin{equation*}
		D_\theta^h\nabla v=\nabla  D_\theta^hv,\quad \int_\Om vD_\theta^{-h}u\df z=-\int_\Om uD_\theta^h v\df z, 
	\end{equation*}  then 
	\begin{equation*}
		\begin{split}
			A_1
			&=\int_\Om \nabla u\cdot A\nabla \left[-D_\theta^h (\zeta^2D_\theta^hu)\right]\df z\\
			&=\int_\Om (\pt_{\theta}u )\pt_{\theta}\left[-D_\theta^h (\zeta^2D_\theta^hu)\right]\df z+\int_\Om r^\al(\pt_ru)\pt_r\left[-D_\theta^h (\zeta^2D_\theta^hu)\right]\df z, 
		\end{split}
	\end{equation*}
	i.e., 
	\begin{equation*}
		\begin{split}
			A_1
			&=\int_\Om (D_\theta^h D_{\theta}u)\pt_{\theta}(\zeta^2D_\theta^hu)\df z+\int_\Om r^\al (D_\theta^h \pt_ru)\pt_r(\zeta^2D_\theta^hu)\df z\\
			&=\int_\Om \zeta^2\left(  (D_\theta^h \pt_{\theta}u)^2+r^\al (D_\theta^h \pt_ru)^2\right)\df z +2\int_\Om \zeta (\pt_{\theta}\zeta)(D_\theta^h \pt_{\theta}u)D_\theta^h u\df z\\
			&\geq \f{1}{2}\int_\Om \zeta^2\left(  (D_\theta^h \pt_{\theta}u)^2+r^\al (D_\theta^h \pt_ru)^2\right)\df z-4\int_\Om (\pt_{\theta}\zeta)^2(D_\theta^hu)^2\df z. 
		\end{split}
	\end{equation*}
	This, together with \cite[Theorem 3 (i)  (p. 292) in Chapter 5.8.2]{Evans}, we get
	\begin{equation*}
		\begin{split}
			A_1\geq   \f{1}{2}\int_\Om \zeta^2\left( (D_\theta^h \pt_{\theta}u)^2+r^\al (D_\theta^h \pt_ru)^2\right)\df z-C\int_\Om |\pt_\theta u|^2\df z, 
		\end{split}
	\end{equation*}
	where the constant $C>0$ is absolute. 
	
	Now, from \cite[Theorem 3 (i)  (p. 292) in Chapter 5.8.2]{Evans}, we get
	\begin{equation*}
		\begin{split}
			A_2
			&\leq \f{1}{2}\int_\Om g^2\df z+\f{1}{2}\int_\Om v^2\df z\leq \f{1}{2}\int_\Om g^2\df z+C\int_\Om \left|\pt_{\theta}\left(\zeta^2D_\theta^hu\right)\right|^2\df z\\
			&\leq \f{1}{2}\int_\Om g^2\df z+C\int_\Om |\pt_{\theta}u|^2\df z+\f{1}{4}\int_\Om \zeta^2 \left(D_\theta^h \pt_{\theta}u\right)^2\df z. 
		\end{split}
	\end{equation*}
	Then 
	\begin{equation*}
		\int_\Om \zeta^2 \left(  (D_\theta^h\pt_{\theta}u)^2+r^\al (D_\theta^h \pt_ru)^2\right)\df z\leq C\int_\Om |\pt_\theta u|^2\df z+C\int_\Om g^2\df z, 
	\end{equation*}
	and hence
	\begin{equation}\label{02.11.6}
		\begin{split}
			\int_\Om \zeta^2\left(  (\pt_{\theta\theta}u)^2+r^\al (\pt_{\theta r}u)^2\right)\df z\leq C\int_\Om g^2\df z, 
		\end{split}
	\end{equation}
	by \eqref{02.11.2} and \cite[Theorem 3 (ii)  (p. 292) in Chapter 5.8.2]{Evans}, where the constant $C>0$ depending only on $\al$. This implies 
	\begin{equation*}
		\pt_{\theta\theta}u\in L^2(\Om) \mbox{ and } r^\f{\al}{2}\pt_{\theta r} u\in L^2(\Om). 
	\end{equation*}
	
	Finally, from \eqref{02.11.6} and  
	\begin{equation*}
		-\pt_{\theta\theta}u-\pt_r(r^\al \pt_ru)=g
	\end{equation*}
	we get \eqref{02.11.5} and $\pt_r(r^\al \pt_ru)\in L^2(\Om)$ and  
	\begin{equation*}
		r^{1+\f{\al}{2}}\pt_{rr} u\in L^2(\Om), 
	\end{equation*}
	by 
	\begin{equation*}
		r^{1-\f{\al}{2}}\pt_r(r^\al \pt_ru)=\al r^\f{\al}{2} \pt_ru+r^{1+\f{\al}{2}}\pt_{rr}u. 
	\end{equation*}
	This completes the proof of this theorem.
\end{proof}

The following lemma will also be used in the analysis of the degenerate hyperbolic equation \eqref{01.21.1}.

\begin{lemma}\label{02.02.L1}
	Let $u$ be the weak solution of \eqref{01.22.2} with respect to $g$. Then we have  
	
	(i) $\pt_\theta u=0$ on $\Ga$, 
	
	(ii) for any $\e\in [0,\f{2-\al}{4})$, 
	\begin{equation*}
		r^\f{\al}{2}u^2\in W^{1,1}(\Om), \mbox{ and } r^{\f{\al}{2}+\e} u\in H^1(\Om), \mbox{ and } r u^2=ru=0 \mbox{ on } \Ga^*, 
	\end{equation*} 
	
	(iii) and 
	\begin{equation*}
		r(\nabla u\cdot A\nabla u)\in W^{1,1}(\Om), \mbox{ and } r(\nabla u\cdot A\nabla u)=0 \mbox{ on }\Ga^*. 
	\end{equation*}
\end{lemma}

\begin{proof}
	We prove this lemma by the following steps. 
	
	{\it Step 1}. We prove (i).
	
	Since $u=0$ on $\Ga=\mbT\ts \{1\}$, then $\pt_\theta u=0$ on $\Ga$. This proves (i). 
	
	{\it Step 2}. We prove (ii). 
	
	Let $u\in H_\Ga^1(\Om;w)$. From $\int_\Om r^\f{\al}{2}u^2\df z\leq \int_\Om u^2\df z$, and from Lemma \ref{01.04.L1} we obtain 
	\begin{equation*}
		\begin{split}
			\int_\Om \left|\nabla (r^\f{\al}{2}u^2)\right|\df z
			\leq C\int_\Om r^{\f{\al}{2}-1}u^2\df z+C\int_\Om \nabla u\cdot A\nabla u\df z\leq C\int_\Om \nabla u\cdot A\nabla u\df z, 
		\end{split}
	\end{equation*}
	where the constants $C>0$ depending only $\al$, 
	hence $r^\f{\al}{2}u^2\in W^{1,1}(\Om)$ and $ru^2=r^{1-\f{\al}{2}}(r^\f{\al}{2}u^2)=0$ on $\Ga^*$.

	Since $\int_\Om (r^{\f{\al}{2}+\e}u)^2\df z\leq \int_\Om u^2\df z$ with $\e\in [0,\f{2-\al}{4})$, and from Lemma \ref{01.04.L1} we get
	\begin{equation*}
		\begin{split} 
			\int_\Om \left|\nabla (r^{\f{\al}{2}+\e}u)\right|^2\df z
			&\leq C\int_\Om  (\pt_{\theta}u)^2\df z+C\int_\Om r^\al(\pt_ru)^2\df z+C\int_\Om u^2r^{\al+2\e-2}\df z\\
			&\leq C\int_\Om \nabla u\cdot A\nabla u\df z, 
		\end{split} 
	\end{equation*}
	we obtain $r^{\f{\al}{2}+\e}u\in H^1(\Om)$. Hence $ru=r^{1-\f{\al}{2}-\e}(r^{\f{\al}{2}+\e}u)=0$ on $\Ga^*$. 
	
	Since $\int_\Om r(\pt_{\theta}u)^2\df z\leq \int_\Om (\pt_{\theta}u)^2\df z$, and  
	\begin{equation*}
		\int_\Om \left|\pt_{\theta}\left[r(\pt_{\theta}u)^2\right]\right|\df z\leq \int_\Om (\pt_{\theta}u)^2\df z+\int_\Om (\pt_{\theta\theta}u)^2\df z, 
	\end{equation*}
	and 
	\begin{equation*}
		\int_\Om \left|\pt_r\left[r(\pt_{\theta}u)^2\right]\right|\df z\leq 2\int_\Om (\pt_{\theta}u)^2\df z+\int_\Om r^\al (\pt_{\theta r}u)^2\df z, 
	\end{equation*}
	these imply $r(\pt_{\theta}u)^2\in W^{1,1}(\Om)$  by Theorem \ref{02.01.T1}. 
	
	{\it Step 3}. We prove (iii). 
	
	Since $\int_\Om r^{1+\al}(\pt_ru)^2\df z\leq \int_\Om r^\al (\pt_r u)^2\df z$, and 
	\begin{equation*}
		\int_\Om \left|\pt_{\theta}\left[r^{1+\al}(\pt_ru)^2\right]\right|\df z\leq \int_\Om r^\al (\pt_ru)^2\df z+\int_\Om r^\al (\pt_{\theta r}u)^2\df z, 
	\end{equation*}
	and 
	\begin{equation*}
		\begin{split}
			\int_\Om \left|\pt_r\left[r^{1+\al}(\pt_ru)^2\right]\right|\df z\leq (2+\al)\int_\Om r^\al (\pt_ru)^2\df z+\int_\Om r^{2+\al}(\pt_{rr}u)^2\df z. 
		\end{split}
	\end{equation*}
	This implies that $r^{1+\al}(\pt_ru)^2\in W^{1,1}(\Om)$. From above, we get $r(\nabla u\cdot A\nabla u)\in W^{1,1}(\Om)$. 
	
	Since  $\pt_{\theta}u\in L^2(\Om)$, and from Theorem \ref{02.01.T1} we know
	\begin{equation*}
		\pt_{\theta}(\pt_{\theta}u)=\pt_{\theta\theta} u\in L^2(\Om), \mbox{ and } r^\f{\al}{2}\pt_r(\pt_{\theta }u)=r^\f{\al}{2}\pt_{\theta r}u\in L^2(\Om), 
	\end{equation*}
	and from (i) we get $\pt_{\theta}u\in H_\Ga^1(\Om;w)$.  Hence $r(\pt_{\theta}u)^2=0$ on $\Ga^*$ by (ii) (replace $u$ by  $\pt_\theta u$ in (ii)). 
	
	Finally, we prove $r^{1+\al} (\pt_ru)^2=0$ on $\Ga^*$. 
	
	Since $\int_\Om (r\pt_ru)^2\df z\leq \int_\Om r^\al (\pt_ru)^2\df z$, and 
	\begin{equation*}
		\int_\Om \left(\pt_{\theta} [r\pt_ru]\right)^2\df z\leq \int_\Om r^\al (\pt_{\theta r}u)^2\df z, 
	\end{equation*}
	and 
	\begin{equation*}
		\int_\Om r^\al (\pt_r[r\pt_ru])^2\df z\leq 2\int_\Om r^\al (\pt_ru)^2\df z+2\int_\Om r^{2+\al}(\pt_{rr}u)^2\df z, 
	\end{equation*}
	then $r\pt_r u\in H^1(\Om;w)$. Taking $\zeta=\zeta(r)\in C^\iy(\ol\R), 0\leq \zeta\leq 1$ such that 
	\begin{equation*}
		\zeta=1 \mbox{ on } \left(0,\f{1}{2}\right), \quad \zeta=0 \mbox{ on } \left(\f{3}{4},1\right), \quad |\zeta'|\leq C, 
	\end{equation*}
	where the constant $C>0$ is absolute. Hence we get 
	\begin{equation*}
		\begin{split} 
			-\int_{\Ga_\de^*}r^{1+\al} (\pt_ru)^2\df S
			&=g(\de; A\nabla u, \zeta r\pt_ru)\\
			&=\lg \ga_\nu (A\nabla u), \zeta r\pt_ru\rg_{H^{-\f{1}{2}}(\Ga_\de^*  ), H^\f{1}{2}(\Ga_\de^*  )}\ra 0 \mbox{ as } \de\ra 0^+
		\end{split} 
	\end{equation*}
	by $u\in H^2(\Om;w)$ and  Lemma \ref{01.22.L5} and Theorem \ref{01.22.T1} and \eqref{02.11.7}.  Therefore, $r^{1+\al}(\pt_ru)^2=0$ on $\Ga^*$. We complete the proof of this lemma. 
\end{proof}

\subsection{Spectrum}

\begin{notation}\label{01.22.N1}
Now, from Lemma \ref{01.21.L1}, the partial differential operator $\mcA$ enjoys a discrete spectrum
\begin{equation*}
	0<\la_1<\la_2\leq \la_3\leq \cdots \ra +\iy. 
\end{equation*}
i.e., $\la_n\ (n\in\N^*)$ is the solution of the following equation 
\begin{equation}\label{01.22.3}
	\begin{cases}
		\mcA\Phi_n=\la_n\Phi_n, &\mbox{in }\Om,\\
		\Phi_n=0, &\mbox{on }\Ga,\\
		\f{\pt\Phi_n}{\pt\nu_A}=0, &\mbox{on }\Ga^*,  
	\end{cases}
\end{equation}
and $\Phi_n\in D(\mcA)$ is called the eigenfunction of $\mcA$ with respect to the eigenvalue $\la_n$. We denote $\Phi_n\ (n\in\N^*)$ is the orthonormal basis of $L^2(\Om)$. See \cite[Theorem 7 in Appendix D.6, p. 728]{Evans}. 
\end{notation}

\begin{lemma}\label{06.28.L1}
	Let $u=\sum_{i=1}^\iy u_i \Phi_i\in H_\Ga^1(\Om;w)$ with $u_i=(u,\Phi_i)_{L^2(\Om)}$ for all $i\in\N^*$. We have $\nabla u=\sum_{i=1}^\iy u_i\nabla \Phi_i$ and $\|u\|_{H_\Ga^1(\Om;w)}=(\sum_{i=1}^\iy u_i^2\la_i)^\f{1}{2}$, and
	\begin{equation*}
		u\in H^2(\Om;w)\Lra \sum_{i=1}^\iy u_i^2\la_i^2<\iy,
	\end{equation*}
	and
	\begin{equation*}
		\mcA u=\sum_{i=1}^\iy u_i\la_i\Phi_i, \mbox{ and } \|\mcA u\|_{L^2(\Om)}=\left(\sum_{i=1}^\iy u_i^2\la_i^2\right)^\f{1}{2}.
	\end{equation*}
\end{lemma}

\begin{proof}
	From \eqref{01.22.3}, from $\Phi_n\in D(\mcA)=H^2(\Om;w)\cap H_\Ga^1(\Om;w)$ and Theorem \ref{01.22.T2} (ii),  we have 
	\begin{equation}\label{06.04.3}
		\int_\Om \nabla \Phi_k\cdot A\nabla \Phi_l\df z=\de_{kl}\la_k \mbox{ for } k,l\in \N^*, 
	\end{equation} 
	where $\de_{kl}$ is the Kronecker delta function, i.e., $\de_{kl}=1$ for $k=l$ and $\de_{kl}=0$ for $k\neq l$. 
	
	We now demonstrate that $\{\la_k^{-\f{1}{2}}\Phi_k\}_{k=1}^\iy$ forms an orthonormal basis of $H_\Ga^1(\Om;w)$.
	
	From \eqref{06.04.3}, it is straightforward to verify that $\{\la_k^{-\f{1}{2}}\Phi_k\}_{k=1}^\iy$ is an orthonormal subset of $H_\Ga^1(\Om;w)$. To prove it is a basis, assume by contradiction that there exists $0\neq u\in H_\Ga^1(\Om;w)$ such that
	\begin{equation*}
		(u,\Phi_k)_{H_\Ga^1(\Om;w)}=\int_\Om \nabla \Phi_k\cdot A\nabla u\df z=0 \mbox{ for all }k\in\N^*. 
	\end{equation*}
	Given that $\{\Phi_k\}_{k\in\N}$ is an orthonormal basis of $L^2(\Om)$, for $u\in H_\Ga^1(\Om;w)$, we can express
	\begin{equation}\label{06.04.4}
		u=\sum_{k=1}^\iy d_k\Phi_k,  \mbox{ where }  d_k=(u, \Phi_k)_{L^2(\Om)}, k\in\N.
	\end{equation}
	Then, from \eqref{01.22.3}, we have
	\begin{equation*}
		0=(u,\Phi_k)_{H_\Ga^1(\Om;w)}=\int_\Om \nabla \Phi_k\cdot A\nabla u\df z=\la_k\int_\Om \Phi_ku\df z =\la_kd_k,
	\end{equation*}
	which implies $d_k=0$. This leads to $u=0$, a contradiction. 
	
	From above, since $u\in H_\Ga^1(\Om;w)$,  we have 
	\begin{equation*} 
		u=\sum_{i=1}^\iy e_i\left(\la_i^{-\f{1}{2}}\Phi_k\right) \mbox{ in } H_\Ga^1(\Om;w), \mbox{ and } \|u\|_{H_\Ga^1(\Om;w)}^2=\sum_{i=1}^\iy e_i^2,
	\end{equation*}
	and 
	\begin{equation*}
		e_i=\int_\Om \nabla u\cdot A\left(\la_i^{-\f{1}{2}}\nabla\Phi_i\right)\df z=\la_i^{-\f{1}{2}}\int_\Om \nabla u\cdot A\nabla \Phi_i\df z=\la_i^\f{1}{2}\int_\Om u\Phi_i\df z=\la_i^\f{1}{2}u_i,
	\end{equation*} 
	hence,  $\nabla u=\sum_{i=1}^\iy u_i\nabla\Phi_i$. Moreover, $\|u\|_{H_\Ga^1(\Om;w)}=(\sum_{i=1}^\iy \la_iu_i^2)^\f{1}{2}$. 
	
	Let $u\in H^2(\Om;w)$. Take $\vp_n=\sum_{i=1}^n u_i\Phi_k$ for each $n\in\N^*$, then $u\in D(\mcA)$, and from $\Phi_i\in D(\mcA)\ (i\in\N^*)$ and Theorem \ref{01.22.T2} (ii), we get 
	\begin{equation*}
		\begin{split}
			(\mcA u, \mcA\vp_n)_{L^2(\Om)}
			&=\sum_{i=1}^n u_i\la_i \int_\Om (\mcA u)\Phi_i\df z=\sum_{i=1}^n u_i\la_i\int_\Om \nabla u\cdot A\nabla\vp\df z=\sum_{i=1}^n u_i\la_i\int_\Om u\mcA\Phi_i\df z\\
			&=\sum_{i=1}^n u_i\la_i^2\int_\Om u\Phi_i\df z=\sum_{i=1}^n u_i^2\la_i^2
		\end{split}
	\end{equation*}
	and $\|\mcA\vp_n\|_{L^2(\Om)}^2=\sum_{i=1}^n u_i^2\la_i^2$ we obtain
	\begin{equation*}
		\sum_{i=1}^nu_i^2\la_i^2\leq \|\mcA u\|_{L^2(\Om)}^2
	\end{equation*}
	for all $n\in\N$ by Cauchy inequality. This implies that $\sum_{i=1}^\iy u_i^2\la_i^2\leq \|\mcA u\|_{L^2(\Om)}^2<\iy$.

	Let $\sum_{i=1}^\iy u_i^2\la_i^2<\iy$. For each  $\vp\in C_0^\iy(\Om)$,  we have
	\begin{equation*}
		\begin{split}
			(\mcA u, \vp)_{L^2(\Om)}
			&=\int_\Om \nabla u\cdot A\nabla\vp\df z=\sum_{i=1}^\iy u_i \int_\Om \nabla\Phi_i\cdot A\nabla\vp\df z=\sum_{i=1}^\iy u_i\la_i\int_\Om \Phi_i\vp\df z
		\end{split}
	\end{equation*}
	in the sense of distributions. 
	Note that
	\begin{equation*}
		\int_\Om \left(\sum_{i=1}^n u_i\la_i\Phi_i\right)^2\df z=\sum_{i=1}^n u_i^2\la_i^2\leq \sum_{i=1}^\iy u_i^2\la_i^2<\iy \mbox{ for all } n\in\N,
	\end{equation*}
		i.e., $\sum_{i=1}^\iy u_i\la_i\Phi_i\in L^2(\Om)$. Hence
	\begin{equation*}
		(\mcA u, \vp)_{L^2(\Om)}=\left(\sum_{i=1}^\iy u_i\la_i\Phi_i, \vp\right)_{L^2(\Om)}.
	\end{equation*}
		This implies that $\mcA u=\sum_{i=1}^\iy u_i\la_i\Phi_i\in L^2(\Om)$, and
		hence
		\[
		\|\mcA u\|_{L^2(\Om)}\leq \left(\sum_{i=1}^\iy u_i^2\la_i^2\right)^\f{1}{2}.
		\]
		We complete the proof of this lemma.
\end{proof}

\subsection{Existence and uniqueness of weak solutions to equation \eqref{01.21.1}}

In this subsection, we establish the existence and uniqueness of weak solutions to equation \eqref{01.21.1}.

\begin{definition}\label{01.06.D1}
	Let $\vp^0\in H_\Ga^1(\Om;w),\vp^1\in L^2(\Om)$ and $f\in L^2(Q)$. 
	A function 
	\begin{equation*}
		\vp\in L^2(0,T; H_\Ga^1(\Om;w))\cap H^1(0,T; L^2(\Om))\cap H^2(0,T; H_\Ga^{-1}(\Om;w))
	\end{equation*}
	is a weak solution of the equation \eqref{01.21.1} with respect to $(\vp^0,\vp^1,f)$, provided
	
	(i) for each $v\in H_\Ga^1(\Om;w)$ and a.e. $t\in [0,T]$, we have
	\begin{equation*}
		\lg \pt_{tt}\vp, v\rg_{H_\Ga^{-1}(\Om;w), H_\Ga^1(\Om;w)}+(\vp, v)_{H_\Ga^1(\Om;w)}=(f,v)_{L^2(\Om)}, 
	\end{equation*}
	
	(ii) $\vp(0)=\vp^0$ and $\pt_t\vp(0)=\vp^1$. 
\end{definition}

The following lemma is \cite[Lemma 2.3 (p. 61)]{Bellassoued}. 

\begin{lemma}\label{11.30.L2}
	Let $(\mcM,g)$ be a $C^m$-Riemannian manifold with compact boundary $\pt \mcM$. Then there exists a $C^{m-1}$-vector field $\bs{n}$ such that 
	\begin{equation*}
		\bs{n}(x)=\nu(x), \ x\in\pt \mcM, \mbox{ and } |\bs{n}(x)|\leq 1, \ x\in \mcM, 
	\end{equation*}
	where $\nu$ is the unit outward normal vector to $\pt \mcM$. 
\end{lemma}

We are now in a position to establish the existence of weak solutions to equation \eqref{01.21.1}.

\begin{theorem}\label{01.06.T1}
	Under Assumption \ref{Assumption (H)}, the equation \eqref{01.21.1} with respect to $(\vp^0,\vp^1,f)$ has a unique weak solution 
	\begin{equation*}
		\begin{split}
			\vp\in L^2(0,T; H_\Ga^1(\Om;w))\cap H^1(0,T; L^2(\Om))\cap H^2(0,T; H_\Ga^{-1}(\Om;w))
		\end{split}
	\end{equation*} 
	satisfies the following estimate
	\begin{equation}\label{01.13.3}
		\begin{split} 
			&\esssup_{t\in [0,T]}\left(\|\vp(t)\|_{H_\Ga^1(\Om;w)}+\|\pt_t\vp(t)\|_{L^2(\Om)}\right)+\|\pt_{tt}\vp\|_{L^2(0,T; H_\Ga^{-1}(\Om;w))}+\left\|\f{\pt\vp}{\pt \nu}\right\|_{L^2(0,T; L^2(\Ga))}\\
			&\leq C\left(\|\vp^0\|_{H_\Ga^1(\Om;w)}+\|\vp^1\|_{L^2(\Om)}+\|f\|_{L^2(Q)}\right),
		\end{split} 
	\end{equation}
	where the constant $C>0$ depending only on $\al$ and $T$. 
	
	Moreover, if we further assume $\vp^0\in D(\mcA), \vp^1\in  H_\Ga^1(\Om;w)$ and $f\in H^1(0,T; L^2(\Om))$, then the weak solution of \eqref{01.21.1} with respect to $(\vp^0,\vp^1,f)$ satisfy
	\begin{equation*}
		\vp\in L^2(0,T; D(\mcA))\cap H^1(0,T; H_\Ga^1(\Om;w))\cap H^2(0,T; L^2(\Om)), 
	\end{equation*}
	and we have the estimate 
	\begin{equation}\label{01.13.20}
		\begin{split}
			&\esssup_{t\in [0,T]}\left(\|\vp(t)\|_{H^2(\Om_{\f{1}{2}})}+\|\vp(t)\|_{D(\mcA)}+\|\pt_t\vp(t)\|_{H_\Ga^1(\Om;w)}+\|\pt_{tt}\vp(t)\|_{L^2(\Om)}\right)\\
			&\hspace{4.5mm}+\left\|\f{\pt\vp}{\pt\nu}\right\|_{H^1(0,T; L^2(\Ga))}\\
			&\leq C\left(\|\vp^0\|_{D(\mcA)}+\|\vp^1\|_{H_\Ga^1(\Om;w)}+\|f\|_{H^1(0,T; L^2(\Om))}\right),
		\end{split}
	\end{equation}
	where the constant $C>0$ depending only on $\al$ and $T$. 
\end{theorem}

\begin{proof}
	We prove this theorem by the following steps. 
	
	{\it Step 1}. Galerkin's method.
	
	Let $2\leq k\in \N^*$. Define
	\begin{equation}\label{01.13.8}
		\vp^k=\sum_{n=1}^k \vp_n^k(t)\Phi_n(x), \mbox{ and } f^k=\sum_{n=1}^k f_n^k(t)\Phi_n(x).
	\end{equation}
	where $\vp_n^k(t), n=1,\cdots, k$ is the solution of the following system
	\begin{equation}\label{01.13.4}
		\f{\df^2}{\df t^2}\vp_n^k(t)+\la_n\vp_n^k(t)=f_n^k(t), \ n=1,\cdots, k
	\end{equation}
	with 
	\begin{equation}\label{01.13.5}
		\vp_n^k(0)=(\vp^0,\Phi_n)_{L^2(\Om)},\ \f{\df \vp_n^k}{\df t}(0)=(\vp^1,\Phi_n)_{L^2(\Om)},\ f_n^k=(f,\Phi_n)_{L^2(\Om)}, \ n=1,\cdots, k. 
	\end{equation}
	It is well-known that the system \eqref{01.13.4} and \eqref{01.13.5} has a unique solution $(\vp_1^k,\cdots,\vp_k^k)$ for $t\in [0,T]$. We note that \eqref{01.13.4} is equivalent to the following equality
	\begin{equation}\label{01.13.6}
		(\pt_{tt}\vp^k, \Phi_n)_{L^2(\Om)}+(\vp^k, \Phi_n)_{H_\Ga^1(\Om;w)}=(f^k,\Phi_n)_{L^2(\Om)}, \ n=1,\cdots, k. 
	\end{equation}
	
	{\it Step 2}. Energy estimate.
	
	Multiplying \eqref{01.13.6} by $\f{\df}{\df t}\vp_n^k(t)$, summing $n=1,\cdots,k$, then 
	\begin{equation*}
		(\pt_{tt}\vp^k,\pt_t\vp^k)_{L^2(\Om)}+(\vp^k, \pt_t\vp^k)_{H_\Ga^1(\Om;w)}=(f^k,\pt_t\vp^k)_{L^2(\Om)}, \mbox{ for a.e.}\ t\in [0,T]. 
	\end{equation*}
	Note that 
	\begin{equation*}
		(\pt_{tt}\vp^k,\pt_t\vp^k)_{L^2(\Om)}=\f{1}{2}\f{\df }{\df t}\|\pt_t\vp^k\|_{L^2(\Om)}^2,\quad (\vp^k,\pt_t\vp^k)_{H_\Ga^1(\Om;w)}=\f{1}{2}\f{\df}{\df t}\|\vp^k\|_{H_\Ga^1(\Om;w)}^2,
	\end{equation*}
	we get
	\begin{equation*}
		\begin{split}
			\f{\df}{\df t}\left(\|\pt_t\vp^k\|_{L^2(\Om)}^2+\|\vp^k\|_{H_\Ga^1(\Om;w)}^2\right)\leq \|f^k\|_{L^2(\Om)}^2+\|\pt_t\vp^k\|_{L^2(\Om)}^2, 
		\end{split}
	\end{equation*}
	and then 
	\begin{equation}\label{01.13.7}
		\begin{split} 
			\|\pt_t\vp^k\|_{L^2(\Om)}^2+\|\vp^k\|_{H_\Ga^1(\Om;w)}^2
			&\leq e^t\left(\int_0^t\|f^k\|_{L^2(\Om)}^2\df t+\|\pt_t\vp^k(0)\|_{L^2(\Om)}^2+\|\vp^k(0)\|_{H_\Ga^1(\Om;w)}^2\right)\\
			&\leq e^T\left(\|f\|_{L^2(Q)}^2+\|\vp^1\|_{L^2(\Om)}^2+\|\vp^0\|_{H_\Ga^1(\Om;w)}^2\right) 
		\end{split} 
	\end{equation}
	for all $t\in [0,T]$ by the Gronwall's inequality and Lemma \ref{06.28.L1}. 
	
	For each $v\in H_\Ga^1(\Om;w)$ with $ \|v\|_{H_\Ga^1(\Om;w)}\leq 1$, we write $v=v^1+v^2$ with $v^1\in \Span\{\Phi_n\}_{n=1}^k$ and $(v^2,\Phi_n)_{L^2(\Om)}=0$ for all $n=1,\cdots,k$, then, from  \eqref{01.13.8} and \eqref{01.13.6}, we get
	\begin{equation*}
		\begin{split}
			\llg \pt_{tt}\vp^k, v\rrg_{H_\Ga^{-1}(\Om;w),H_\Ga^1(\Om;w)}=(\pt_{tt}\vp^k,v^1)_{L^2(\Om)}=(f^k,v^1)_{L^2(\Om)}-(\vp^k, v^1)_{H_\Ga^1(\Om;w)}. 
		\end{split}
	\end{equation*}
	Which together with \eqref{01.13.7} we obtain 
	\begin{equation*}
		\|\pt_{tt}\vp^k\|_{H_\Ga^{-1}(\Om;w)}\leq C\left(\|f^k\|_{L^2(\Om)}+\|\vp^k\|_{H_\Ga^1(\Om;w)}\right)
	\end{equation*}
	by Remark \ref{01.21.R1} and $\|v^1\|_{H_\Ga^1(\Om;w)}\leq 1$, where the constant $C>0$ depends only on $\al$. Hence, from \eqref{01.13.7}, we get
	\begin{equation}\label{01.13.9}
		\|\pt_{tt}\vp^k\|_{L^2(0,T; H_\Ga^{-1}(\Om;w))}\leq C\left(\|f\|_{L^2(Q)}+\|\vp^1\|_{L^2(\Om)}+\|\vp^0\|_{H_\Ga^1(\Om;w)}\right), 
	\end{equation}
	where the constant $C>0$ depends only on $\al$ and $T$. 
	
	{\it Step 3}. Approximation, or weak convergence. 
	
	From \eqref{01.13.7} and \eqref{01.13.9}, there exists
	\begin{equation*}
		\psi\in L^2(0,T; H_\Ga^1(\Om;w))\cap H^1(0,T; L^2(\Om))\cap H^2(0,T; H_\Ga^{-1}(\Om;w))
	\end{equation*} 
	such that 
	\begin{equation}\label{01.13.10}
		\begin{split}
			\vp^k
			&\ra \psi \mbox{ weak star in } L^\iy(0,T; H_\Ga^1(\Om;w)), \\
			\pt_t\vp^k
			&\ra \pt_t\psi \mbox{ weak star in } L^\iy(0,T; L^2(\Om)),\\
			\pt_{tt}\vp^k
			&\ra \pt_{tt}\psi \mbox{ weakly in } L^2(0,T; H_\Ga^{-1}(\Om;w))
		\end{split}
	\end{equation}
	by abstract subsequence. Hence
	\begin{equation}\label{01.13.11}
		\begin{split}
			&\esssup_{t\in [0,T]}\left(\|\pt_t\psi\|_{L^2(\Om)}+\|\psi\|_{H_\Ga^1(\Om;w)}\right) +\|\pt_{tt}\psi \|_{L^2(0,T; H_\Ga^{-1}(\Om;w))}\\
			&\leq C\left(\|f\|_{L^2(Q)}+\|\vp^1\|_{L^2(\Om)}+\|\vp^0\|_{H_\Ga^1(\Om;w)}\right), 
		\end{split}
	\end{equation}
	where the constant $C>0$ depends only on $\al$ and $T$. Moreover, we have
	\begin{equation}\label{01.13.12}
		\psi\in C([0,T]; L^2(\Om)),\mbox{ and } \pt_t\psi \in C([0,T]; H_\Ga^{-1}(\Om;w)). 
	\end{equation}
	
	{\it Step 4}. Existence of weak solution to equation \eqref{01.21.1}. 
	
	We show $\psi$ is a solution of \eqref{01.21.1} with respect to $(\vp^0,\vp^1,f)$. 
	
	Let $2\leq l\in\N^*$. Taking $h\in C^1([0,T]; H_\Ga^1(\Om;w))$ as the form
	\begin{equation}\label{01.13.14}
		h=\sum_{n=1}^l h_n(t)\Phi_n(x), \mbox{ with } h_n  \mbox{ is a smooth function on } [0,T], 
	\end{equation}
	multiplying $h_n(t)$ on the both sides of \eqref{01.13.6}, summing $n=1,\cdots, l$, integrating on $[0,T]$,  then, for all $k\geq l$, we have
	\begin{equation}\label{01.13.16}
		\begin{split}
			\int_0^T\llg \pt_{tt}\vp^k, h\rrg_{H_\Ga^{-1}(\Om;w),H_\Ga^1(\Om;w)}\df t+\int_0^T (\vp^k, h)_{H_\Ga^1(\Om;w)}\df t=\int_0^T (f^k,h)_{L^2(\Om)}\df t. 
		\end{split}
	\end{equation}
	Which together with \eqref{01.13.10} we get
	\begin{equation}\label{01.13.13}
		\int_0^T\lg \pt_{tt}\psi, h\rg_{H_\Ga^{-1}(\Om;w), H_\Ga^1(\Om;w)}\df t+\int_0^T(\psi, h)_{H_\Ga^1(\Om;w)}\df t=\int_0^T(f,h)_{L^2(\Om)}\df t. 
	\end{equation}
	Note that the set of the functions $h$   defined in \eqref{01.13.14} is dense in $L^2(0,T; H_\Ga^1(\Om;w))$, hence \eqref{01.13.13} also hold for all $h\in L^2(0,T; H_\Ga^1(\Om;w))$. For each $t\in (0,T)$, for each $\e\in \f{1}{4}\min\{t, T-t\}$, choose $\zeta_\de\in C_0^\iy(0,T)\ (0<\de<\e)$ such that 
	\begin{equation*}
		\zeta_\de=1 \mbox{ on } (t-\e,t+\e), \quad \zeta_\de=0 \mbox{ on } (t-\e-\de,t+\e+\de),
	\end{equation*}
	then, for each $v\in H_\Ga^1(\Om;w)$, we have $\zeta_\de v\in L^2(0,T; H_\Ga^1(\Om;w))$ and 
	\begin{equation*} 
		\int_0^T\zeta_\de\lg \pt_{tt}\psi, v\rg_{H_\Ga^{-1}(\Om;w), H_\Ga^1(\Om;w)}\df t+\int_0^T\zeta_\de(\psi, v)_{H_\Ga^1(\Om;w)}\df t=\int_0^T\zeta_\de(f,v)_{L^2(\Om)}\df t. 
	\end{equation*}
	by \eqref{01.13.13}. Letting $\de\ra 0$, we get
	\begin{equation*}
		\int_{t-\e}^{t+\e}\lg\pt_{tt}\psi, v\rg_{H_\Ga^{-1}(\Om;w), H_\Ga^1(\Om;w)}\df t+\int_{t-\e}^{t+\e}(\psi, v)_{H_\Ga^1(\Om;w)}\df t=\int_{t-\e}^{t+\e}(f,v)_{L^2(\Om)}\df t. 
	\end{equation*}
	From the Lebesgue point theorem, letting $\e\ra 0$, we get 
	\begin{equation*}
		\lg\pt_{tt}\psi, v\rg_{H_\Ga^{-1}(\Om;w), H_\Ga^1(\Om;w)}+(\psi, v)_{H_\Ga^1(\Om;w)}=(f,v)_{L^2(\Om)} \mbox{ for a.e.}\  t\in [0,T].
	\end{equation*} 
	This proves Definition \ref{01.06.D1} (i). 
	
	Next, we verify $\psi(0)=\vp^0$ and $\pt_t\psi(0)=\vp^1$. 
	
	For any $h\in C^2([0,T]; H_\Ga^1(\Om;w))$ with $h(T)=\pt_th(T)=0$, from \eqref{01.13.13} and \eqref{01.13.12}, we have
	\begin{equation}\label{01.13.15}
		\begin{split}
			&\iint_Q\psi \pt_{tt}h\df z\df  t+\iint_Q \nabla \psi\cdot A\nabla h\df z\df t\\
			&=\iint_Q fh\df z\df t-(\psi(0),\pt_th(0))_{L^2(\Om)}+\llg \pt_t\psi(0),h(0)\rrg_{H_\Ga^{-1}(\Om;w),H_\Ga^1(\Om;w)}. 
		\end{split}
	\end{equation}
	Similarly from \eqref{01.13.16} we get
	\begin{equation*}
		\begin{split}
			&\iint_Q\vp^k \pt_{tt}h\df z\df  t+\iint_Q \nabla \vp^k\cdot A\nabla h\df z\df t\\
			&=\iint_Q fh\df z\df t-(\vp^k(0),\pt_th(0))_{L^2(\Om)}+(\pt_t\vp^k, h(0))_{L^2(\Om)}, 
		\end{split}
	\end{equation*}
	this together with \eqref{01.13.5} and \eqref{01.13.10} we obtain 
	\begin{equation*}
		\begin{split}
			&\iint_Q\psi \pt_{tt}h\df z\df  t+\iint_Q \nabla \psi\cdot A\nabla h\df z\df t\\
			&=\iint_Q fh\df z\df t-(\vp^0,\pt_th(0))_{L^2(\Om)}+(\vp^1,h(0))_{L^2(\Om)}. 
		\end{split}
	\end{equation*}
	Which together with \eqref{01.13.15} we deduce 
	\begin{equation*}
		\psi(0)=\vp^0, \mbox{ and } \pt_t\psi(0)=\vp^1. 
	\end{equation*}
	This prove Definition \eqref{01.06.D1} (ii).

	{\it Step 5}. Uniqueness.
	
	We only need to show that the weak solution to equation \eqref{01.21.1} with respect to $(0,0,0)$ is zero function. 
	Let $\vp$ be the solution of \eqref{01.21.1} with respect to $(\vp^0,\vp^1,f)=(0,0,0)$. Fix $s\in [0,T]$, taking 
	\begin{equation*}
		h(t)=
		\begin{cases}
			\int_t^s\vp(\tau)\df\tau, & t\in [0,s],\\
			0, &t\in [s,T], 
		\end{cases}
	\end{equation*}
	then $h(t)\in H_\Ga^1(\Om;w)$ for a.e. $t\in [0,T]$, and from Definition \ref{01.06.D1} (i) we have 
	\begin{equation*}
		\int_0^s\llg\pt_{tt}\vp, h\rrg_{H_\Ga^{-1}(\Om;w),H_\Ga^1(\Om;w)}\df t+\int_0^sw\nabla \vp \cdot \nabla h\df t=0. 
	\end{equation*}
	Note that $\pt_t\vp(0)=h(s)=0$, integrating by parts with respect to $t$,  we get 
	\begin{equation*}
		\int_0^s(\pt_t\vp,\vp)_{L^2(\Om)}\df t-\int_0^s(\pt_th,h)_{H_\Ga^1(\Om;w)}\df t=0
	\end{equation*}
	by $\pt_th=-\vp\ (0\leq t\leq s)$, hence, 
	\begin{equation*}
		\f{1}{2}\int_0^s\f{\df}{\df t}\|\vp(t)\|_{L^2(\Om)}^2\df t-\f{1}{2}\int_0^s\f{\df}{\df t}\|h\|_{H_\Ga^1(\Om;w)}^2\df t=0. 
	\end{equation*}
	Therefore, from $\vp(0)=0$ and $h(s)=0$, we get
	\begin{equation*}
		\|\vp(s)\|_{L^2(\Om)}^2+\|h(0)\|_{H_\Ga^1(\Om;w)}^2=\|\vp(0)\|_{L^2(\Om)}^2+\|h(s)\|_{H_\Ga^1(\Om;w)}^2=0. 
	\end{equation*}
	i.e., $\vp(s)=0$. This shows that $\vp=0$ by the arbitrary of $s\in [0,T]$. 
	
	From the uniqueness and \eqref{01.13.11} we get the  1th-3th estimate in \eqref{01.13.3}. 
	
	{\it Step 6}. Improved regularity.
	
	Since $f\in H^1(0,T; L^2(\Om))$, from   \eqref{01.13.4} we get
	\begin{equation*}
		\f{\df^2}{\df t^2}\left(\f{\df\vp_n^k}{\df t}\right)+\la_n\f{\df\vp_n^k}{\df t}(t)=\f{\df f_n^k}{\df t}(t), \ n=1,\cdots, k.
	\end{equation*}
	Denote $\wt \vp^k=\pt_t\vp^k$, 
	from  \eqref{01.13.6}, we have
	\begin{equation*}
		(\pt_{tt}\wt\vp^k, \Phi_n)_{L^2(\Om)}+(\wt\vp^k, \Phi_n)_{H_\Ga^1(\Om;w)}=(\pt_tf^k,\Phi_n)_{L^2(\Om)}, \ n=1,\cdots, k. 
	\end{equation*}
	Multiplying $\f{\df^2\vp_n^k}{\df t^2}(t)$, summing $n=1,\cdots,k$, we discover
	\begin{equation*}
		(\pt_{tt}\wt\vp^k, \pt_t\wt \vp^k)_{L^2(\Om)}+(\wt \vp^k, \pt_t\wt\vp^k)_{H_\Ga^1(\Om;w)}=(\pt_tf^k, \pt_t\wt\vp_n^k)_{L^2(\Om)}. 
	\end{equation*}
	Similar to \eqref{01.13.7} in Step 2, we obtain 
	\begin{equation*}
		\begin{split}
			&\esssup_{t\in [0,T]}\left(\|\pt_t\wt \vp^k\|_{L^2(\Om)}^2+\|\wt\vp^k\|_{H_\Ga^1(\Om;w)}^2\right) \\
			&\leq e^T\left(\|\pt_tf^k\|_{L^2(Q)}^2+\|\pt_t\wt\vp^k(0)\|_{L^2(\Om)}^2+\|\wt\vp^k(0)\|_{H_\Ga^1(\Om;w)}^2\right). 
		\end{split}
	\end{equation*}
	Note that from \eqref{01.13.5} we have  $\wt\vp^k(0)=\pt_t\vp^k(0)=\sum_{n=1}^k (\vp^1,\Phi_n)_{L^2(\Om)}\Phi_n$, then, from  Lemma \ref{06.28.L1}, we get
	\begin{equation*}
		\|\wt\vp^k(0)\|_{H_\Ga^1(\Om;w)}^2=\sum_{n=1}^k (\vp^1,\Phi_n)_{L^2(\Om)}^2\la_n\leq \sum_{n=1}^\iy(\vp^1,\Phi_n)_{L^2(\Om)}^2\la_n= \|\vp^1\|_{H_\Ga^1(\Om;w)};
	\end{equation*}  and from \eqref{01.13.8} and  \eqref{01.13.4} and Lemma \ref{06.28.L1} we have 
	\begin{equation*}
		\begin{split}
			\pt_t\wt\vp^k(0)=\pt_{tt}\vp^k(0)
			&=\sum_{n=1}^k \f{\df^2\vp_n^k}{\df t^2}(0)\Phi_n=\sum_{n=1}^k f_n^k(0)\Phi_n-\sum_{n=1}^k\vp_n^k(0)\la_n\Phi_n 
		\end{split}
	\end{equation*}
	and 
	\begin{equation*}
		\|\pt_t\wt\vp^k(0)\|_{L^2(\Om)}^2\leq 2\|f^k(0)\|_{L^2(\Om)}^2+2\|\vp^0\|_{D(\mcA)}^2\leq C\|f\|_{H^1(0,T; L^2(\Om))}^2+2\|\vp^0\|_{D(\mcA)}^2, 
	\end{equation*}
	where the second inequality we use  \cite[Theorem 2 (iii) in Chapter 5.9.2, p. 302]{Evans}, and the constant $C>0$ depends only on $T$. Hence, 
	\begin{equation}\label{01.13.17}
		\begin{split} 
			&\esssup_{t\in [0,T]}\left(\|\pt_t\wt \vp^k\|_{L^2(\Om)}^2+\|\wt\vp^k\|_{H_\Ga^1(\Om;w)}^2\right) \\
			&\leq C\left(\|f\|_{H^1(0,T; L^2(\Om))}^2+\|\vp^0\|_{D(\mcA)}^2+\|\vp^1\|_{H_\Ga^1(\Om;w)}^2\right)
		\end{split} 
	\end{equation}
	by Lemma \ref{06.28.L1}, where the constant $C>0$ depends only on $T$. 
	
	Multiplying $\la_n\vp_n^k(t)$ on the both sides of \eqref{01.13.6}, summing $n=1,\cdots, k$, note that  $\vp^k,\mcA\vp^k=\sum_{n=1}^k \vp_n^k(t) \la_n\Phi_n\in H^2(\Om;w)\cap H_\Ga^1(\Om;w)$ for a.e. $t\in [0,T]$ by \eqref{01.13.8},  from Theorem \ref{01.22.T2} (ii),  we get
	\begin{equation*}
		\begin{split}
			\|\mcA\vp^k\|_{L^2(\Om)}^2=(\vp^k, \mcA\vp^k)_{H_\Ga^1(\Om;w)}=(f^k-\pt_{tt}\vp^k,\mcA \vp^k)_{L^2(\Om)}, 
		\end{split}
	\end{equation*}
	then, from \eqref{01.13.17} and \cite[Theorem 2 (iii) in Chapter 5.9.2, p. 302]{Evans}, we deduce
	\begin{equation}\label{01.13.18}
		\begin{split} 
			\|\mcA\vp^k(t)\|_{L^2(\Om)}
			&\leq 2\|f^k(t)\|_{L^2(\Om)}+2\|\pt_{tt}\vp^k(t)\|_{L^2(\Om)}\\
			&\leq C\left(\|f\|_{H^1(0,T; L^2(\Om))}+\|\vp^0\|_{D(\mcA)}+\|\vp^1\|_{H_\Ga^1(\Om;w)}\right), 
		\end{split} 
	\end{equation}
	where the constant $C>0$ depends only on $T$. 
	
	Choosing 
	\begin{equation}\label{01.13.21}
		\zeta\in C_0^\iy(\R^N), \ 0<\zeta<1,
	\end{equation}
	such that 
	\begin{equation*}
		\zeta=0 \mbox{ on } \mbT\ts \left(0,\f{1}{4}\right), \quad \zeta=1 \mbox{ on } \mbT\ts \left(\f{1}{2},1\right),
	\end{equation*}
	and 
	\begin{equation*}
		|\nabla\zeta|\leq C  \mbox{ and } \left|\f{\pt^2\zeta}{\pt x_i\pt x_j}\right|\leq C \mbox{ on }\R^N,
	\end{equation*}
	where the constants $C>0$ are absolute. 
	Note that $z=\zeta\vp^k$ is the solution of the following equation 
	\begin{equation*}
		\begin{cases} 
			\mcA z=\zeta f^k-\zeta\pt_{tt}\vp^k-\vp^k\mcA\zeta-2\nabla\zeta\cdot A\nabla \vp^k,  &\mbox{in } \mbT\ts \left(\f{1}{4},1\right),\\
			z=0, &\mbox{on }\mbT\ts \left\{\f{1}{4},1\right\}, 
		\end{cases}
	\end{equation*}
	this is a uniformly elliptic equation on $\mbT\ts (\f{1}{4},1)$, hence, from \cite[Theorem 1 (p. 327)  in Chapter 6.3.1 and Theorem 4 (p. 334) in Chapter 6.3.2]{Evans} and \eqref{01.13.3} and \eqref{01.13.7}, we have 
	\begin{equation}\label{01.13.19}
		\begin{split} 
			\|\vp^k(t)\|_{H^2(\Om_{\f{1}{2}})}
			&\leq C\left(\|f^k(t)\|_{L^2(\Om)}+\|\pt_{tt}\vp^k(t)\|_{L^2(\Om)}+\|\vp^k(t)\|_{H_\Ga^1(\Om;w)}\right)\\
			&\leq C\left(\|f\|_{H^1(0,T; L^2(\Om))}+\|\vp^0\|_{D(\mcA)}+\|\vp^1\|_{H_\Ga^1(\Om;w)}\right), 
		\end{split} 
	\end{equation}
	where the constants $C>0$ depend only on $\al$ and $T$. 
	This, combing   \eqref{01.13.10} and $\vp=\psi$,  we have proved the 1th-4th estimate in \eqref{01.13.20}.

	{\it Step 7}. Hidden regularity. 
	
	We will prove the 4th estimate in \eqref{01.13.3} and the 5th estimate in \eqref{01.13.20}.  
	
	Let $\vp^0\in D(\mcA), \vp^1\in H_\Ga^1(\Om;w)$ and $f\in H^1(0,T; L^2(\Om))$. Since $\vp$ is the weak solution of \eqref{01.21.1} with respect to $(\vp^0,\vp^1,f)$, we obtain $\Psi=\pt_t\vp\in L^2(0,T; H_\Ga^1(\Om;w))$ is the weak solution of the following system
	\begin{equation*}
		\begin{cases}
			\pt_{tt}\Psi+\mcA\Psi=\pt_tf, &\mbox{in }Q, \\
			\Psi=0, &\mbox{on } \Ga\ts (0,T), \\
			\f{\pt \Psi}{\pt\nu_A}=0, &\mbox{on }\Ga^*\ts (0,T), \\
			\Psi(0)=\vp^1, \pt_t\Psi(0)=f(0)-\mcA\vp^0, &\mbox{in }\Om
		\end{cases}
	\end{equation*}
	by \eqref{01.13.20}. Suppose that the 4th estimate in \eqref{01.13.3} holds, then the 5th estimate in \eqref{01.13.20} also follows. So, we only need to prove the 4th estimate in \eqref{01.13.3}.  
	
	To prove the 4th estimate in \eqref{01.13.3}, by the density argument, we only need to prove the case: $y^0\in D(\mcA), y^1\in H_\Ga^1(\Om;w)$ and $f\in H^1(0,T; L^2(\Om))$. 
	
	Let $\zeta$ be defined in \eqref{01.13.21}. Multiplying $\zeta\bs{n}\cdot \nabla \vp$ on the both sides of \eqref{01.21.1}, from \eqref{01.13.20}, integrating on $Q$ (indeed, on $(\mbT\ts (\f{1}{4},1))\ts (0,T)$), integration by parts, we have
	\begin{equation*}
		\begin{split}
			A_1+A_2
			&=\iint_Q (\pt_{tt}\vp)(\zeta\bs{n}\cdot \nabla \vp)\df z\df t-\iint_Q [\Div(A\nabla \vp)](\zeta\bs{n}\cdot \nabla \vp)\df z\df t\\
			&=\iint_Q f(\zeta\bs{n}\cdot \nabla \vp)\df z\df t=A_3. 
		\end{split}
	\end{equation*}
	From $\pt_t\vp=0$ on $\Ga\ts (0,T)$ and $\zeta=0$ on $\mbT\ts [0,\f{1}{4}]$ and the 1th-2th estimate in \eqref{01.13.3}, we have
	\begin{equation*}
		\begin{split}
			A_1
			&=\int_\Om (\pt_t\vp)(\zeta\bs{n}\cdot \nabla\vp)\df z\bigg|_{t=0}^{t=T}+\f{1}{2}\iint_Q (\pt_t\vp)^2\Div(\zeta\bs{n})\df z\df t\\
			&\leq C\esssup_{t\in [0,T]}\left(\|\pt_t\vp(t)\|_{L^2(\Om)}^2+\|\vp(t)\|_{H_\Ga^1(\Om;w)}^2\right)+C\|\pt_t\vp\|_{L^2(Q)}^2\\
			&\leq C\left(\|\vp^0\|_{H_\Ga^1(\Om;w)}^2+\|\vp^1\|_{L^2(\Om)}^2+\|f\|_{L^2(Q)}^2\right);
		\end{split}
	\end{equation*}
	from Lemma \ref{11.30.L2} we have (with $\bs{n}=(n_\theta,n_r)$)
	\begin{equation*}
		\begin{split}
			A_2
			&=-\f{1}{2}\iint_{\Ga\ts (0,T)}\left(\f{\pt \vp}{\pt \nu}\right)^2\df S\df t+\iint_Q (A\nabla\vp)\cdot [(D(\zeta\bs{n}))\nabla \vp]\df z\df t\\
			&\hspace{4.5mm}-\f{1}{2}\iint_Q (\nabla\vp\cdot A\nabla\vp)\Div(\zeta \bs{n})\df z\df t-\f{\al}{2}\iint_Q \zeta n_rr^{\al-1}(\pt_ru)^2\df z\df t;
		\end{split}
	\end{equation*}
	from the 1th-2th estimate in \eqref{01.13.3}, we have
	\begin{equation*}
		\begin{split}
			A_3\leq \|f\|_{L^2(Q)}^2+C\|\vp\|_{L^2(0,T; H_\Ga^1(\Om;w))}^2\leq C\left(\|\vp^0\|_{H_\Ga^1(\Om;w)}^2+\|\vp^1\|_{L^2(\Om)}^2+\|f\|_{L^2(Q)}^2\right),
		\end{split}
	\end{equation*}
	hence, 
	\begin{equation}\label{02.11.1}
		\iint_{\Ga\ts (0,T)}\left(\f{\pt\vp}{\pt \nu}\right)^2\df S\df t\leq C\left(\|\vp^0\|_{H_\Ga^1(\Om;w)}^2+\|\vp^1\|_{L^2(\Om)}^2+\|f\|_{L^2(Q)}^2\right), 
	\end{equation}
	where the constant $C>0$ depending only on $\al$ and $T$. This proves the 4th estimate in \eqref{01.13.3}. 
	
	{\it Step 8}. From Step 1-Step 6 and Step 7, we get \eqref{01.13.3}. From Step 6 and Step 7, we get \eqref{01.13.20}. We complete the proof of this theorem. 
\end{proof}

The next lemma concerns the regularity of weak solutions of \eqref{01.21.1}.

\begin{lemma}\label{01.23.L1}
	Let $\vp^0\in D(\mcA)$ and $\vp^1\in H_\Ga^1(\Om;w)$ and $f\in H^1(0,T; L^2(\Om))$, and  $\vp$ be the weak solution of \eqref{01.21.1} with respect to $(\vp^0,\vp^1,f)$. Then  
	\begin{equation}\label{02.11.4}
		\begin{split} 
		r (\pt_t\vp)^2=0 \mbox{ on }\Ga^*\ts (0,T).
		\end{split} 
	\end{equation} 
\end{lemma}

\begin{proof}
		From the third term in \eqref{01.13.20} and the boundary condition
		$\pt_t\vp=0$ on $\Ga\ts (0,T)$, we obtain
		$\pt_t\vp\in L^2(0,T; H_\Ga^1(\Om;w))$ and
	\begin{equation*}
		\begin{split}
			&\left|\iint_{\Ga_{\de_1}^*\ts (0,T)} r (\pt_t\vp)^2\df S\df t-\iint_{\Ga_{\de_2}^*\ts (0,T)} r (\pt_t\vp)^2\df S\df t\right|\\
			&=\left|\iint_{Q_{\de_1}-Q_{\de_2}}\pt_r\left[r(\pt_t\vp)^2\right]\df S\df t\right|=\left|\iint_{Q_{\de_1}-Q_{\de_2}}(\pt_t\vp)^2\df z\df t+2\iint_{Q_{\de_1}-Q_{\de_2}}r (\pt_t\vp)(\pt_t\pt_r\vp)\df z\df t\right|\\
			&\leq C \iint_{Q_{\de_1}-Q_{\de_2}}(\pt_t\vp)^2\df z\df t+C\iint_{Q_{\de_1}-Q_{\de_2}}r^{\al}(\pt_r\pt_t\vp)^2\df z\df t
		\end{split}
	\end{equation*}
	for all $0\leq \de_1\leq \de_2\leq \f{1}{8}$.  This shows that 
	\begin{equation*}
		\iint_{\Ga_\de^*\ts (0,T)} r (\pt_t\vp)^2\df S\df t  
	\end{equation*}
	is continuous for $\de\in [0,\f{1}{8}]$. Now, from $\pt_t\vp=0$ on $\Ga\ts (0,T)$ and Lemma \ref{01.04.L1} and the third term in \eqref{01.13.20}, we get 
	\begin{equation*}
		\begin{split}
			\iint_{\Ga_\de^*\ts (0,T)}r(\pt_t\vp)^2\df S\df t
			&=\de^{1-\f{\al}{2}}\iint_{Q_\de} \pt_r\left(r^\f{\al}{2}(\pt_t\vp)^2\right)\df z\df t\\
			&=\f{\al}{2}\de^{1-\f{\al}{2}}\iint_{Q_\de} r^{\f{\al}{2}-1}(\pt_t\vp)^2\df z\df t+\de^{1-\f{\al}{2}}\iint_{Q_\de} r^\f{\al}{2} (\pt_t\vp)(\pt_r\pt_t\vp)\df z\df t\\
			&\leq C\de^{1-\f{\al}{2}}\iint_{Q_\de} r^\al (\pt_r\pt_t\vp)^2\df z\df t+C\de^{1-\f{\al}{2}}\iint_{Q_\de} (\pt_t\vp)^2\df z\df t\ra 0 
		\end{split}
	\end{equation*}
	as $\de\ra 0$. 
This completes the proof of this lemma. 
\end{proof}

\begin{corollary}\label{02.14.C1}
	Let $\vp^0\in D(\mcA), \vp^1\in H_\Ga^1(\Om;w)$ and $f\in H^1(0,T; L^2(\Om))$. Let $\vp$ be the solution of \eqref{01.21.1} with respect to $(\vp^0,\vp^1,f)$. Then 
	\begin{equation}\label{02.15.2}
		\pt_{\theta \theta }\vp\in L^2(Q), \mbox{ and } r^\f{\al}{2}\pt_{\theta r}\vp\in L^2(Q), \mbox{ and } r^{1+\f{\al}{2}}\pt_{rr}\vp\in L^2(Q).
	\end{equation}
	Moreover, there exists a positive  constant $C$, depending only on $\al$, such that 
	\begin{equation}\label{02.15.1}
		\begin{split} 
		&\iint_Q \left[(\pt_{\theta\theta}\vp)^2+r^\al (\pt_{\theta r}\vp)^2+r^{2+\al}(\pt_{rr}\vp)^2\right]\df z\df t\\
		&\leq C\left(\|\vp^0\|_{D(\mcA)}^2+\|\vp^1\|_{H_\Ga^1(\Om;w)}^2+\|f\|_{H^1(0,T; L^2(\Om))}^2\right). 
		\end{split} 
	\end{equation}
\end{corollary}

\begin{proof}
	Note that $\vp$ satisfies $\pt_{tt}\vp+\mcA \vp=f$, i.e., $\mcA\vp=f-\pt_{tt}\vp\in L^2(\Om)$ for a.e. $t\in [0,T]$.  From Theorem \ref{02.01.T1} we get \eqref{02.15.2}, and from \eqref{02.11.5} we get
	\begin{equation*}
		\begin{split}
			&\iint_Q \left[(\pt_{\theta\theta}\vp)^2+r^\al (\pt_{\theta r}\vp)^2+r^{2+\al}(\pt_{rr}\vp)^2\right]\df z\df t\\
			&\leq C\left(\|\pt_{tt}\vp\|_{L^2(Q)}^2+\|f\|_{L^2(Q)}^2\right)\leq C\left(\|\vp^0\|_{D(\mcA)}^2+\|\vp^1\|_{H_\Ga^1(\Om;w)}^2+\|f\|_{H^1(0,T; L^2(\Om))}^2\right),
		\end{split}
	\end{equation*}
	where we use \eqref{01.13.20} in the last inequality. We complete the proof of this corollary. 
\end{proof}

\begin{corollary}\label{02.15.C2}
	Let $\vp^0\in D(\mcA), \vp^1\in H_\Ga^1(\Om;w)$ and $f\in H^1(0,T; L^2(\Om))$. Let $\vp$ be the weak solution to equation \eqref{01.21.1} with respect to $(\vp^0,\vp^1, f)$. Then we have
	
	(i) $\pt_\theta \vp=0$ on $\Ga\ts (0,T)$, 
	
	(ii) for any $\e\in [0,\f{2-\al}{4})$, we have $r^\f{\al}{2}\vp^2\in L^2(0,T; W^{1,1}(\Om))$, and $r^{\f{\al}{2}+\e}\vp\in L^2(0,T; H^1(\Om))$, and $r \vp^2=r\vp=0 \mbox{ on } \Ga^*\ts (0,T)$, 
	
	(iii) $r(\nabla \vp\cdot A\nabla\vp)\in L^2(0,T; W^{1,1}(\Om))$, and $ r(\nabla \vp\cdot A\nabla \vp)=0 \mbox{ on }\Ga^*\ts (0,T)$. 
\end{corollary}

\begin{proof}
	This is a direct consequence of Lemma \ref{02.02.L1}.
\end{proof}

\begin{corollary}\label{02.12.C1}
	Let $\vp^0\in D(\mcA)$ and $\vp^1\in H_\Ga^1(\Om;w)$ and $f=0$. Let $\vp$ be the weak solution of \eqref{01.21.1} with respect to $(\vp^0,\vp^1,f=0)$. Then, for each $t\in [0,T]$, we have  
	\begin{equation*}
		E(t)=E(0), \mbox{ with } E(t)=\f{1}{2}\int_\Om \left[(\pt_t\vp)^2+\nabla \vp\cdot A\nabla \vp\right]\df z. 
	\end{equation*}
\end{corollary}

\begin{proof}
	Multiplying \eqref{01.21.1} by $\pt_t\vp$, from Theorem \ref{01.06.T1} and Theorem \ref{01.22.T2} (ii), integrating on $Q$, we get
	\begin{equation*}
		\begin{split}
			0
			&=\iint_{Q} \left[\pt_{tt}\vp-\Div(A\nabla\vp)\right]\pt_t\vp\df z\df t=\f{1}{2}\iint_Q \pt_t(\pt_t\vp)^2\df z\df t+\f{1}{2}\iint_Q \pt_t(\nabla\vp\cdot A\nabla\vp)\df z\df t\\
			&=\f{1}{2}\int_\Om\left[ (\pt_t\vp)^2+(\nabla\vp\cdot A\nabla\vp)\right]\df z\bigg|_{t=0}^{t=T}.
		\end{split}
	\end{equation*}
	This proves the corollary. 
\end{proof}

\section{Observability}\label{S3}

Let $\de_0\in (0,\f{1}{32})$ be a given constant. Set
\begin{equation*}
	I_\om=[0, 4\de_0)\cup (2\pi-4\de_0, 2\pi),\qquad \om=I_\om\ts (0,1). 
\end{equation*} 
The interior observation region is therefore the strip $\omega$, and
\begin{equation*}
	\mathbb{T}\setminus I_\om=(4\de_0, 2\pi-4\de_0). 
\end{equation*}
In the argument below, $\mathbb{T}\setminus I_\om$ is treated inside a single
coordinate chart of the smooth manifold $\mathbb{T}$; more precisely, this set is
contained in the chart $(\delta_0,2\pi-\delta_0)$. 

\subsection{A motivating example}

We now explain why the cut-off decomposition introduced below is natural, and
why the strip $\omega$ must appear in the final estimate.

Choose $\chi\in C_0^\infty(0,2)$, $\chi\not\equiv 0$, and for $\varepsilon\in (0,\frac14)$ define
\[
\eta_\varepsilon(r)=\chi(r/\varepsilon).
\]
Then $\eta_\varepsilon$ is supported in $(0,2\varepsilon)$, so that
\[
\eta_\varepsilon(1)=0,\qquad \eta_\varepsilon'(1)=0.
\]
For each integer $n\geq 1$, consider the oscillatory profile
\[
u_{n,\varepsilon}(\theta,r,t)=\eta_\varepsilon(r)\sin n(\theta-t).
\]
This family is not intended as a literal counterexample to the homogeneous
observability statement. Instead, it captures the quasimode-type obstruction
that motivates the present proof.

One has
\[
\partial_{tt}u_{n,\varepsilon}=-n^2\eta_\varepsilon(r)\sin n(\theta-t),
\qquad
\partial_{\theta\theta}u_{n,\varepsilon}=-n^2\eta_\varepsilon(r)\sin n(\theta-t),
\]
hence the time and angular contributions cancel exactly:
\[
\partial_{tt}u_{n,\varepsilon}-\partial_{\theta\theta}u_{n,\varepsilon}=0.
\]
Therefore
\[
\partial_{tt}u_{n,\varepsilon}-\Div(A\nabla u_{n,\varepsilon})
=-\partial_r\!\left(r^\alpha \eta_\varepsilon'(r)\right)\sin n(\theta-t).
\]
Since $\eta_\varepsilon$ is supported near $r=0$, this residual is concentrated near the degenerate boundary. On the other hand, the top-boundary observation is completely blind to this family:
\[
u_{n,\varepsilon}(\theta,1,t)=0,\qquad
\partial_r u_{n,\varepsilon}(\theta,1,t)=\eta_\varepsilon'(1)\sin n(\theta-t)=0.
\]
Thus this family exhibits the basic obstruction mechanism: strong angular
oscillation with negligible radial penetration can remain essentially invisible
from the top boundary while the defect is confined near the degenerate side. In
this sense, $u_{n,\varepsilon}$ should be read as a source-supported quasimode
mechanism rather than as an exact homogeneous counterexample.

There is a second obstruction. The multiplier used below is
\[
H(\theta,r)=((\theta-\pi),r),
\]
which is not globally defined on the periodic variable $\theta\in \mathbb{T}$.
Hence the multiplier method cannot be applied directly on the whole cylinder.
To overcome this topological difficulty, we remove a narrow angular strip
\[
I_\om=[0,4\delta_0)\cup(2\pi-4\delta_0,2\pi),
\]
and choose a cut-off $\zeta=\zeta(\theta)$ such that $\zeta\equiv 1$ away from $I_\om$ and $\zeta\equiv 0$ near $I_\om$. We then decompose the solution as
\[
\vp=\psi+\xi,\qquad \psi=\zeta\vp,\qquad \xi=(1-\zeta)\vp.
\]

The roles of these two pieces are transparent:
\begin{itemize}
	\item $\psi$ is supported in a single chart of $\mathbb{T}$, so the multiplier argument with $H=((\theta-\pi),r)$ becomes legitimate;
	\item $\xi$ is supported in the strip $\om=I_\om\ts(0,1)$, which is exactly where the localization error produced by the cut-off is concentrated.
\end{itemize}

Therefore the strip $\omega$ is not an artificial add-on. It is forced by the
geometry of the quasimode-type obstruction and by the necessity of localizing
the multiplier argument on the periodic variable. This is why the final
observability inequality proved below takes a mixed form, involving the
boundary term on $\Gamma$ and the interior observation on $\omega$.

\subsection{Cut-off decomposition}

Taking $\zeta=\zeta(\theta)\in C^\iy(\ol\R), 0\leq \zeta\leq 1$ such that 
\begin{equation}\label{03.16.1}
	\zeta=1 \mbox{ on }(3\de_0,2\pi-3\de_0), \quad \zeta=0 \mbox{ on }(-\iy,2\de_0)\cup (2\pi-2\de_0,+\iy), 
\end{equation}
and
\begin{equation*}
	|\zeta'|<C\de_0^{-1},\quad |\zeta''|<C\de_0^{-2}, 
\end{equation*}
where the positive constants $C$ are absolute. 

The localized component $\psi=\zeta \vp$ solves the following system:
\begin{equation}\label{03.01.1}
	\begin{cases}
		\pt_{tt}\psi-\Div(A\nabla\psi)=-2\nabla\zeta \cdot A\nabla\vp-\vp\Div(A\nabla\zeta)\equiv g, &\mbox{in }\wh Q, \\
		\psi=0, &\mbox{on } \pt \wh Q-\wh\Ga^*\ts (0,T), \\
		\f{\pt \psi}{\pt \nu_A}=0, &\mbox{on } \wh \Ga^*\ts (0,T),\\
		\psi(0)=\zeta \vp^0, \pt_t\psi(0)=\zeta\vp^1, &\mbox{in }\wh \Om, 
	\end{cases}
\end{equation}
where $\wh \Om=(\de_0,2\pi-\de_0)\ts (0,1)$, $\wh Q=\wh \Om\ts (0,T)$, and $\wh\Ga^*=(\de_0,2\pi-\de_0)\ts \{0\}$. The remainder $\xi=(1-\zeta)\vp$ solves
\begin{equation}\label{03.22.2}
	\begin{cases}
		\pt_{tt}\xi -\Div(A\nabla\xi)=2\nabla\zeta \cdot A\nabla\vp+\vp\Div(A\nabla\zeta)=-g, &\mbox{in }\wt Q,\\
		\xi=0, &\mbox{on }\pt \wt Q-\wt \Ga^*\ts (0,T),\\
		\f{\pt \xi}{\pt \nu_A}=0, &\mbox{on }\wt \Ga^*\ts (0,T),\\
		\xi(0)=(1-\zeta)\vp^0, \pt_t\xi(0)=(1-\zeta)\vp^1, &\mbox{in }\wt \Om, 
	\end{cases}
\end{equation}
where $\wt \Om=([0,4\de_0)\cup ( 2\pi-4\de_0,2\pi))\ts (0,1)$, $\wt Q=\wt \Om\ts (0,T)$, and $\wt \Ga^*=([0,4\de_0)\cup ( 2\pi-4\de_0,2\pi))\ts \{0\}$. Thus $\wt \Om=\om$.

Since $f=0\in H^1(0,T; L^2(\Om))$, Theorem \ref{01.06.T1} gives
\begin{equation*}
	\begin{split}
		\vp\in L^2(0,T; D(\mcA))\cap H^1(0,T; H_\Ga^1(\Om;w))\cap H^2(0,T; L^2(\Om)), 
	\end{split}
\end{equation*}
and hence
\begin{equation}\label{03.01.4}
	\psi\in L^2(0,T; D(\mcA))\cap H^1(0,T; H_\Ga^1(\Om;w))\cap H^2(0,T; L^2(\Om)). 
\end{equation}
Moreover, Corollary \ref{02.14.C1} yields
\begin{equation}\label{03.17.4}
	\pt_{\theta\theta}\psi, r^\f{\al}{2}\pt_{\theta r}\psi, r^{1+\f{\al}{2}}\pt_{rr}\psi\in L^2( Q), 
\end{equation}
and then
\begin{equation}\label{03.01.5}
	\begin{split}
		\pt_\theta \psi\in H_\Ga^1( \Om;w), \mbox{ and }  r^\al \pt_r\psi\in H^1(\Om;w). 
	\end{split}
\end{equation} 

Recall \eqref{03.16.1}, we get 
\begin{equation}\label{03.16.2}
	\psi=0 \mbox{ on } \left[(\de_0, 2\de_0)\ts (0,1)\ts (0,T)\right]\cup \left[(2\pi-2\de_0, 2\pi-\de_0)\ts (0,1)\ts (0,T)\right]. 
\end{equation}
 
\subsection{Preliminary identities}

In this subsection we record the elementary identities needed in the multiplier computation. They are consequences of the support properties of $\psi$, the regularity stated in \eqref{03.01.4}--\eqref{03.01.5}, and repeated applications of Fubini's theorem.

\begin{lemma}\label{03.19.L1}
The following identities hold:
\end{lemma}

\begin{proof}

Since $\psi\in L^2(0,T; H_\Ga^1(\Om;w))$, we first obtain:

(I) $\psi^2\in L^1(\wh Q)$, and $\pt_\theta \psi^2=2\psi\pt_\theta \psi\in L^1(\wh Q)$, which shows that $\psi^2\in C_0(\de_0, 2\pi-\de_0)$ for a.e.~$(r,t)\in (0,1)\ts (0,T)$ by \eqref{03.16.2}, and 
\begin{equation}\label{03.19.1}
	\begin{split}
		\iint_{\wh Q} \pt_\theta \psi^2\df z\df t=\int_0^T\int_0^1\int_{\de_0}^{2\pi-\de_0}\pt_\theta \psi^2\df \theta \df r\df t=0
	\end{split}
\end{equation}
by Fubini's Theorem and \eqref{03.16.2}; 

(II) $\psi^2\in L^1(\wh Q)$, and $\pt_r (r\psi^2)=\psi^2+2r\psi \pt_r\psi\in L^1(\wh Q)$, which shows that $r\psi^2\in C_0[0,1]$ for a.e.~$(\theta,t)\in (\de_0,2\pi-\de_0)\ts (0,T)$ by Corollary \ref{02.15.C2} (ii), and 
\begin{equation}\label{03.19.2}
	\begin{split}
		\iint_{\wh Q} \pt_r(r\psi^2)\df z\df t=\int_0^T\int_{\de_0}^{2\pi-\de_0}\int_0^1 \pt_r(r\psi^2)\df r\df \theta \df t=0
	\end{split}
\end{equation}
by Fubini's Theorem and Corollary \ref{02.15.C2}. 

Next, since $\pt_\theta \psi\in H_\Ga^1(\Om;w)$, we have:

(III) $(\pt_\theta\psi)^2\in L^1(\wh Q)$, and $\pt_\theta (\pt_\theta\psi)^2=2\pt_{\theta}\psi \pt_{\theta\theta}\psi\in L^1(\wh Q)$, which shows that $(\pt_\theta\psi)^2\in C_0(\de_0,2\pi-\de_0)$ for a.e.~$(r,t)\in (0,1)\ts (0,T)$ by \eqref{03.16.2}, and 
\begin{equation}\label{03.19.3}
	\iint_{\wh Q} \pt_\theta (\pt_\theta\psi)^2\df z\df t=\int_0^T\int_0^1\int_{\de_0}^{2\pi-\de_0}\pt_\theta (\pt_\theta \psi)^2\df \theta \df r\df t=0
\end{equation}
by Fubini's Theorem and \eqref{03.16.2}; 

(IV) $(\pt_\theta\psi)^2\in L^1(\wh Q)$, and $\pt_r[r(\pt_\theta \psi)^2]=(\pt_\theta\psi)^2+2r(\pt_\theta \psi)(\pt_{r\theta}\psi)\in L^1(\wh Q)$ by \eqref{03.17.4}, which shows that $r(\pt_\theta \psi)^2\in C_0[0,1]$ for a.e.~$(\theta,t)\in (\de_0,2\pi-\de_0)\ts (0,T)$ by Corollary \ref{02.15.C2} (i) and  (iii), and 
\begin{equation}\label{03.19.4}
	\begin{split}
		\iint_{\wh Q} \pt_r \left[r(\pt_\theta \psi)^2\right]\df z\df t=\int_0^T\int_{\de_0}^{2\pi-\de_0}\int_0^1\pt_r \left[r(\pt_\theta\psi)^2\right]\df \theta \df r\df t=0
	\end{split}
\end{equation}
by Fubini's Theorem and Corollary \ref{02.15.C2} (i) and (iii) and $\al\in [1,2)$. 

(V) Since $r^\al (\pt_r\psi)^2\in L^1(\wh Q)$ and $\pt_\theta [r^\al (\pt_r\psi)^2]=2r^\al (\pt_r\psi)(\pt_{\theta r}\psi)\in L^1(\wh Q)$ by \eqref{03.17.4}, which shows that 
$r^\al (\pt_r\psi)^2\in C_0(\de_0,2\pi-\de_0)$ for a.e.~$(r,t)\in (0,1)\ts (0,T)$, and 
\begin{equation}\label{03.19.5}
	\begin{split}
		\iint_{\wh Q}\pt_\theta \left[r^\al (\pt_r\psi)^2\right]\df z\df t=\int_0^T\int_0^1\int_{\de_0}^{2\pi-\de_0}\pt_\theta \left[r^\al (\pt_r\psi)^2\right]\df \theta \df r\df t=0
	\end{split}
\end{equation}
by Fubini's Theorem and \eqref{03.16.2}. 

(VI) Since $r^{\al+1}(\pt_r\psi)^2\in L^1(\wh Q)$ and $\pt_r[r^{\al+1}(\pt_r\psi)^2]=(\al+1)r^\al(\pt_r\psi)^2+2r^{\al+1}(\pt_r\psi)(\pt_{rr}\psi)\in L^1(\wh Q)$ by \eqref{03.17.4}, which shows that $r^{\al+1}(\pt_r\psi)^2\in C[0,1]$ for a.e.~$(\theta,t)\in (\de_0,2\pi-\de_0)\ts (0,T)$, and 
\begin{equation}\label{03.19.6}
	\begin{split}
		\iint_{\wh Q}\pt_r\left[r^{\al+1}(\pt_r\psi)^2\right]\df z\df t
		&=\int_0^T\int_{\de_0}^{2\pi-\de_0}\int_0^1\pt_r\left[r^{\al+1}(\pt_r\psi)^2\right]\df r\df \theta \df t\\
		&=\iint_{\Ga\ts (0,T)} (\pt_r\psi)^2\df S\df t
	\end{split}
\end{equation}
by Fubini's Theorem and Corollary \ref{02.15.C2} (iii) and $\al\in [1,2)$. 

\end{proof}

\subsection{Multiplier argument}

Let the multiplier 
\begin{equation}\label{03.02.2}
	z_0=(\pi,0),\quad H(z)=z-z_0=(\theta-\pi,r),\quad H\cdot \nabla\psi=(\theta-\pi)\pt_\theta\psi+r\pt_r\psi. 
\end{equation}

We first verify that the multiplier is admissible. Note that  
\begin{equation*}
	\begin{split}
		\iint_{\wh Q}(H\cdot \nabla\psi)^2\df z\df t
		&\leq 2\iint_{\wh Q} \left[(\theta-\pi)^2(\pt_\theta\psi)^2+r^2(\pt_r\psi)^2\right]\df z\df t\\
		&\leq 2\pi^2\iint_{\wh Q} \left[(\pt_\theta \psi)^2+r^\al (\pt_r\psi)^2\right]\df z\df t<+\iy,
	\end{split}
\end{equation*}
and from 
\begin{equation}\label{03.22.1}
	\begin{split}
		\nabla (H\cdot \nabla\psi)=\left(\pt_\theta \psi+(\theta-\pi)\pt_{\theta\theta}\psi+r\pt_{\theta r}\psi, (\theta-\pi)\pt_{\theta r}\psi+\pt_r\psi+r\pt_{rr}\psi\right)
	\end{split}
\end{equation}
we get 
\begin{equation*}
	\begin{split}
		&\iint_{\wh Q} \left[\nabla (H\cdot \nabla\psi)\cdot A\nabla (H\cdot \nabla\psi)\right]\df z\df t\\
		&=\iint_{\wh Q} \left[\pt_\theta \psi+(\theta-\pi)\pt_{\theta\theta}\psi+r\pt_{\theta r}\psi\right]^2\df z\df t+\iint_{\wh Q} r^\al \left[(\theta-\pi)\pt_{\theta r}\psi+\pt_r\psi+r\pt_{rr}\psi\right]^2\df z\df t\\
		&\leq C\left(\iint_{\wh Q} \left[(\pt_\theta \psi)^2+r^\al (\pt_r\psi)^2+(\pt_{\theta\theta}\psi)^2+r^\al (\pt_{\theta r}\psi)^2+r^{\al+2}(\pt_{rr}\psi)^2\right]\df z\df t\right)<+\iy
	\end{split}
\end{equation*}
by Corollary \ref{02.14.C1}, these show that $H\cdot \nabla\psi\in H^1(\Om;w)$. 

We now multiply \eqref{03.01.1} by $H\cdot \nabla\psi$ and integrate over $\wh Q$. This yields
\begin{equation*}
	\begin{split}
		B_1+B_2
		&\equiv \iint_{\wh Q} (\pt_{tt}\psi)(H\cdot \nabla\psi)\df z\df t-\iint_{\wh Q} [\Div(A\nabla\psi)](H\cdot \nabla\psi)\df z\df t\\
		&=\iint_{\wh Q} g(H\cdot \nabla\psi)\df z\df t\equiv B_3.
	\end{split}
\end{equation*}

We next compute $B_1$. From \eqref{01.13.20}, we have $\pt_t\psi\in L^2(0,T; H_\Ga^1(\Om;w))$, and therefore
\begin{equation*}
	\begin{split}
		B_1
		&=\iint_{\wh Q} \pt_t\left[(\pt_t\psi) (H\cdot \nabla\psi)\right]\df z\df t-\iint_{\wh Q} (\pt_t\psi) H\cdot \nabla\pt_t\psi\df z\df t\\
		&=\int_{\wh Q} (\pt_t\psi)(H\cdot\nabla\psi)\df z\bigg|_{t=0}^{t=T}-\f{1}{2}\iint_{\wh Q} [\nabla (\pt_t\psi)^2]\cdot H\df z\df t\\
		&=\int_{\wh Q} (\pt_t\psi)(H\cdot\nabla\psi)\df z\bigg|_{t=0}^{t=T}+\f{1}{2}\iint_{\wh Q} (\pt_t\psi)^2\Div H\df z\df t\\
		&=\int_{\wh Q} (\pt_t\psi)(H\cdot\nabla\psi)\df z\bigg|_{t=0}^{t=T}+\iint_{\wh Q} (\pt_t\psi)^2\df z\df t
	\end{split}
\end{equation*}
by Lemma \ref{01.23.L1} (or similarly \eqref{03.19.1} and \eqref{03.19.2}). 

We then compute $B_2$. Since $H\cdot\nabla\psi\in H^1(\Om;w)$, Corollary \ref{03.01.C6} applies. Combined with \eqref{03.16.2} and the identity $\pt_\theta \psi=0$ on $\Ga\ts (0,T)$ (see Corollary \ref{02.15.C2} (i)), we obtain
\begin{equation*}
	\begin{split}
		B_2
		&=-\iint_{\Ga\ts (0,T)} (A\nabla\psi\cdot \nu)(H\cdot\nabla\psi)\df S\df t+\iint_{\wh Q} A\nabla\psi\cdot \nabla (H\cdot \nabla\psi)\df z\df t\\
		&=-\iint_{\Ga\ts (0,T)} (\pt_r\psi)^2\df S\df t+\iint_{\wh Q}A\nabla\psi\cdot \nabla(H\cdot \nabla\psi)\df z\df t. 
	\end{split}
\end{equation*}

Using \eqref{03.22.1} together with \eqref{03.19.3}-\eqref{03.19.6}, we further compute
\begin{equation*}
	\begin{split}
		&\iint_{\wh Q}A\nabla\psi\cdot \nabla(H\cdot \nabla\psi)\df z\df t\\
		&=\iint_{\wh Q} \left[(\pt_\theta \psi)^2+r^\al (\pt_r\psi)^2\right]\df z\df t\\
		&\hspace{4.5mm}+\f{1}{2}\iint_{\wh Q} \left[(\theta-\pi)\pt_\theta(\pt_\theta\psi)^2+(\theta-\pi)r^\al\pt_\theta (\pt_r\psi)^2+r\pt_r(\pt_\theta \psi)^2+r^{\al+1}\pt_r(\pt_r\psi)^2\right]\df z\df t\\
		&=\f{1}{2}\iint_{\Ga\ts (0,T)}(\pt_r\psi)^2\df S\df t-\f{\al}{2}\iint_{\wh Q} r^\al(\pt_r\psi)^2\df z\df t. 
	\end{split}
\end{equation*}
Then we get 
\begin{equation*}
	\begin{split}
		B_2=-\f{1}{2}\iint_{\Ga\ts (0,T)}(\pt_r\psi)^2\df S\df t-\f{\al}{2}\iint_{\wh Q} r^\al(\pt_r\psi)^2\df z\df t. 
	\end{split}
\end{equation*}

Combining the formulae for $B_1$ and $B_2$, we arrive at
\begin{equation}\label{03.22.4}
	\begin{split}
		B_1+B_2
		&=\int_{\wh \Om} (\pt_t\psi)(H\cdot\nabla\psi)\df z\bigg|_{t=0}^{t=T}+\iint_{\wh Q} (\pt_t\psi)^2\df z\df t\\
		&\hspace{4.5mm}-\f{1}{2}\iint_{\Ga\ts (0,T)}(\pt_r\psi)^2\df S\df t-\f{\al}{2}\iint_{\wh Q} r^\al(\pt_r\psi)^2\df z\df t. 
	\end{split}
\end{equation} 

We now return to the equation for $\psi$. Multiplying \eqref{03.01.1} by $\pt_t\psi$, and using $\pt_t\psi\in L^2(0,T; H_\Ga^1(\Om;w))$ together with Theorem \ref{01.22.T2} (ii), we get
\begin{equation*}
	\begin{split}
		&\f{1}{2}\int_{\wh \Om}\left[(\pt_t\psi)^2+A\nabla\psi\cdot \nabla\psi\right]\df z\bigg|_{t=0}^{t=T}=\iint_{\wh Q} g(\pt_t\psi)\df z\df t.
	\end{split}
\end{equation*}
i.e., 
\begin{equation}\label{03.23.1}
	\begin{split}
		E_\psi(T)-E_\psi(0)=\iint_{\wh Q}g(\pt_t\psi)\df z\df t, 
	\end{split}
\end{equation}
where 
\begin{equation*}
	E_\psi(s)=\f{1}{2}\int_{\wh\Om}\left[(\pt_t\psi)^2+A\nabla\psi\cdot \nabla\psi\right]\df z\bigg|_{t=s}. 
\end{equation*}
This implies that
\begin{equation*}
	\begin{split}
		\int_0^T E_\psi(t)\df t-TE_\psi(0)
		&=\int_0^T\iint_{\wh \Om\ts (0,t)}g(\pt_t\psi)\df z\df t, 
	\end{split}
\end{equation*}
i.e., 
\begin{equation*}
	\begin{split}
		TE_\psi(0)
		&=\int_0^TE_\psi(t)\df t-\int_0^T\iint_{\wh \Om\ts (0,t)}g(\pt_t\psi)\df z\df t\\
		&=\f{1}{2}\iint_{\wh Q}\left[(\pt_t\psi)^2+A\nabla\psi\cdot\nabla\psi\right]\df z\df t-\int_0^T\iint_{\wh \Om\ts (0,t)}g(\pt_t\psi)\df z\df t. 
	\end{split}
\end{equation*}
Then 
\begin{equation}\label{02.12.3}
	\begin{split}
		&\iint_{\wh Q} (\pt_t\psi)^2\df z\df t-\f{\al}{2}\iint_{\wh Q} r^\al (\pt_r\psi)^2\df z\df t\\
		&=\f{2-\al}{2}TE_\psi(0)+\f{2-\al}{2}\int_0^T\iint_{\wh\Om\ts (0,t)} g(\pt_t\psi)\df z\df t+\f{2+\al}{4}\iint_Q (\pt_t\psi)^2\df z\df t\\
		&\hspace{4.5mm}-\f{2-\al}{4}\iint_Q \nabla\psi\cdot A\nabla\psi \df z\df t-\f{\al}{2}\iint_Q r^\al (\pt_r\psi)^2\df z\df t.
	\end{split}
\end{equation} 

On the other hand, multiplying \eqref{03.01.1} by $\psi$ and integrating over $\wh Q$,
and using $\psi\in L^2(0,T; D(\mcA))$ together with Theorem \ref{01.22.T2} (ii), we get 
\begin{equation*}
	\int_\Om (\pt_t\psi)\psi\df z\bigg|_{t=0}^{t=T}=\iint_Q \left[(\pt_t\psi)^2-\nabla \psi\cdot A\nabla\psi\right]\df z\df t+\iint_{\wh Q} g\psi\df z\df t, 
\end{equation*}
and hence
\begin{equation}\label{02.12.8}
	\begin{split}
		&\f{1}{2}\iint_{\Ga\ts (0,T)}(\pt_r\psi)^2\df S\df t+\iint_{\wh Q} g(H\cdot\nabla\psi)\df z\df t\\
		&\hspace{4.5mm}-\f{2-\al}{2}\int_0^T\iint_{\wh \Om\ts (0,t)} g(\pt_t\psi)\df z\df t+\f{2+\al}{4}\iint_{\wh Q} g\psi\df z\df t\\
		&\geq \f{2-\al}{2}TE_\psi(0)+\int_\Om (\pt_t\psi)\left(H\cdot\nabla\psi+\f{2+\al}{4}\psi \right)\df z\df t\bigg|_{t=0}^{t=T}
	\end{split}
\end{equation}
by  \eqref{03.22.4} and \eqref{02.12.3}.  Denote 
\begin{equation*}
	\begin{split}
		X(t)=\int_\Om (\pt_t\psi)\left(H\cdot\nabla\psi+\f{2+\al}{4}\psi \right)\df z, \mbox{ for } t\in [0,T],  
	\end{split}
\end{equation*}
hence, 
\begin{equation*}
	\begin{split}
		|X(t)|\leq \f{\e}{2}\int_\Om (\pt_t\psi)^2\df z+\f{1}{2\e}\int_\Om \left|H\cdot\nabla\psi+\f{2+\al}{4}\psi\right|^2\df z, \mbox{ for all } \e>0. 
	\end{split}
\end{equation*}
Moreover, from Corollary  \ref{02.15.C2} (ii), and  \eqref{03.19.1} and \eqref{03.19.2}, we get 
\begin{equation*}
	\int_{\wh \Om} (H\cdot\nabla \psi)\psi\df z=\f{1}{2}\int_{\wh \Om} H\cdot \nabla \psi^2\df z=-\int_{\wh \Om} \psi^2\df z, 
\end{equation*}
and therefore
\begin{equation*}
	\begin{split} 
		&\int_{\wh \Om} \left|H\cdot\nabla\psi+\f{2+\al}{4}\psi\right|^2\df z\\
		&=\int_{\wh \Om} (H\cdot\nabla \psi)^2\df z+\f{2+\al}{2}\int_{\wh \Om} (H\cdot\nabla\psi)\psi\df z+\f{(2+\al)^2}{16}\int_{\wh \Om} \psi^2\df z\leq 2\int_{\wh \Om} \nabla\psi\cdot A\nabla\psi\df z
	\end{split} 
\end{equation*}
by $\al\in [1,2)$. Consequently, using $B_1+B_2=B_3$, we get 
\begin{equation*}
	\begin{split}
		&\f{1}{2}\iint_{\Ga\ts (0,T)}(\pt_r\psi)^2\df S\df t+\iint_{\wh Q} g(H\cdot\nabla\psi)\df z\df t\\
		&\hspace{4.5mm}-\f{2-\al}{2}\int_0^T\iint_{\wh \Om\ts (0,t)} g(\pt_t\psi)\df z\df t+\f{2+\al}{4}\iint_{\wh Q} g\psi\df z\df t\\
		&\geq \f{2-\al}{2}TE_\psi(0)-\f{\e}{2}\int_\Om (\pt_t\psi)^2\df z-\f{1}{\e}\int_\Om \nabla \psi\cdot A\nabla\psi\df z\\
		&\geq \left(\f{2-\al}{2}T-\max\left\{\f{\e}{2}, \f{1}{\e}\right\}\right)E_\psi(0). 
	\end{split}
\end{equation*}
Finally, from the definition of $\psi$ (namely $\psi=\zeta \vp$), from
\begin{equation*}
	g=-2\nabla\zeta\cdot A\nabla\vp-\vp\Div(A\nabla\zeta)
	=-2\zeta'(\theta)\pt_\theta\vp-\zeta''(\theta)\vp,
\end{equation*}
and from
\begin{equation*}
	\supp g\s \big([2\de_0,3\de_0)\cup (2\pi-3\de_0,2\pi-2\de_0)\big)\ts (0,1)\ts(0,T)\s \om\ts(0,T),
\end{equation*}
we obtain
\begin{equation*}
	|g|\leq C\left(|\vp|+|\pt_\theta\vp|\right)\chi_{\om}.
\end{equation*}
Moreover, since $\psi=\zeta\vp$ and $H=((\theta-\pi),r)$, we have
\begin{equation*}
	|\psi|+|\pt_t\psi|+|H\cdot\nabla\psi|
	\leq C\left(|\vp|+|\pt_t\vp|+|\pt_\theta\vp|+r|\pt_r\vp|\right)
\end{equation*}
on $\wh Q$. As $\alpha\in[1,2)$ and $0<r<1$, one has $r^2\le r^\alpha$. Hence, by Cauchy--Schwarz and Young's inequality,
\begin{equation*}
	\begin{split}
		\left|\iint_{\wh Q} g(H\cdot\nabla\psi)\df z\df t\right|
		&\leq C\iint_{\om\ts(0,T)}\left(|\vp|+|\pt_\theta\vp|\right)
		\left(|\vp|+|\pt_\theta\vp|+r|\pt_r\vp|\right)\df z\df t\\
		&\leq C\iint_{\om\ts(0,T)}\left[\vp^2+A\nabla\vp\cdot\nabla\vp\right]\df z\df t,
	\end{split}
\end{equation*}
and similarly,
\begin{equation*}
	\left|\int_0^T\iint_{\wh \Om\ts (0,t)} g(\pt_t\psi)\df z\df t\right|
	+\left|\iint_{\wh Q} g\psi\df z\df t\right|
	\leq C\iint_{\om\ts(0,T)}\left[\vp^2+(\pt_t\vp)^2+A\nabla\vp\cdot\nabla\vp\right]\df z\df t.
\end{equation*}
Therefore, combining these bounds with the previous inequality, we get 
\begin{equation}\label{03.22.5}
	\begin{split}
		\left[(2-\al)T-\sqrt{2}\right]E_\psi(0)
		&\leq \iint_{\Ga\ts (0,T)} (\pt_r\psi)^2\df S\df t\\
		&\hspace{4.5mm}+C\iint_{\om\ts (0,T)}\left[\vp^2+(\pt_t\vp)^2+A\nabla\vp\cdot\nabla\vp\right]\df z\df t, 
	\end{split}
\end{equation}
where the positive constant $C$ depends only on $\al$ and $\de_0$. 

It remains to estimate the localized component $\xi$. We multiply
\eqref{03.22.2} by $\pt_t\xi$, where
\[
\pt_t\xi\in L^2(0,T; H_\Ga^1(\wt \Om;w)),
\]
and integrate by parts exactly as in \eqref{03.23.1}. This yields
\begin{equation*}
	\begin{split}
		E_\xi(0)=E_\xi(t)+\iint_{\wt\Om\ts (0,t)} g(\pt_t\xi)\df z\df t, 
	\end{split}
\end{equation*}
where 
\begin{equation*}
	E_\xi(t)=\f{1}{2}\int_{\wt \Om} \left[(\pt_t\xi)^2+A\nabla\xi\cdot\nabla\xi\right]\df z\df t, \mbox{ for } t\in [0,T].
\end{equation*}
Since $\wt Q=\om\ts (0,T)$, we get 
\begin{equation*}
	\begin{split} 
	E_\xi(0)
	&\leq \f{1}{T}\int_0^T E_\xi(t)\df t+\left|\iint_{\wt Q} g(\pt_t\xi)\df z\df t\right|\\
	&\leq \f{C}{T}\iint_{\om\ts (0,T)}\left[\vp^2+(\pt_t\vp)^2+A\nabla\vp\cdot\nabla\vp\right]\df z\df t
	+\left|\iint_{\wt Q} g(\pt_t\xi)\df z\df t\right|.
	\end{split} 
\end{equation*}
Using again the bound
\[
|g|\le C\left(|\vp|+|\pt_\theta\vp|\right)\chi_{\om}
\]
and the identity $\pt_t\xi=(1-\zeta)\pt_t\vp$, we obtain
\begin{equation*}
	\begin{split}
		\left|\iint_{\wt Q} g(\pt_t\xi)\df z\df t\right|
		&\le C\iint_{\om\ts(0,T)}\left(|\vp|+|\pt_\theta\vp|\right)|\pt_t\vp|\df z\df t\\
		&\le C\iint_{\om\ts (0,T)}\left[\vp^2+(\pt_t\vp)^2+A\nabla\vp\cdot\nabla\vp\right]\df z\df t.
	\end{split}
\end{equation*}
Therefore
\begin{equation*}
	\begin{split}
		E_\xi(0)
		&\leq C\left(1+\f{1}{T}\right)\iint_{\om\ts (0,T)}\left[\vp^2+(\pt_t\vp)^2+A\nabla\vp\cdot\nabla\vp\right]\df z\df t, 
	\end{split}
\end{equation*}
where the positive constant $C$ depends only on $\al$ and $\de_0$. 

We now combine the estimates for $\psi$ and $\xi$. Since $\vp=\psi+\xi$, we have
\begin{equation*}
	\begin{split}
		E(0)
		&=\int_\Om \left[(\pt_t\vp)^2+A\nabla\vp\cdot \nabla\vp\right]\df z\bigg|_{t=0}\\
		&=\int_\Om \left[(\pt_t\psi+\pt_t\xi)^2+A(\nabla\psi+\nabla \xi)\cdot (\nabla\psi+\nabla\xi)\right]\df z\bigg|_{t=0} \\
		&\leq 2\int_\Om \left[(\pt_t\psi)^2+(\pt_t\xi)^2+A\nabla\psi\cdot \nabla\psi+A\nabla\xi\cdot \nabla\xi \right]\df z\bigg|_{t=0}=2E_\psi(0)+2E_\xi(0), 
	\end{split}
\end{equation*}
Combining this inequality with \eqref{03.22.5}, we obtain
\begin{equation*}
	\begin{split} 
	&\left[(2-\al)T-\sqrt{2}\right]E(0)\\
	&\leq C\iint_{\Ga\ts (0,T)}(\pt_r\vp)^2\df S\df t+C\iint_{\om\ts (0,T)}\left[\vp^2+(\pt_t\vp)^2+A\nabla\vp\cdot\nabla\vp\right]\df z\df t,
	\end{split} 
\end{equation*}
where the positive constant $C$ depends on $T$, $\al$, and $\de_0$.

\clearpage
\pagestyle{plain}
\markright{}

	\section*{Acknowledgments}
	The authors would like to thank Shugeng Chai (Shanxi University) and Yan He (Hubei University) for helpful discussions related to the observability problem considered in this paper.

\end{document}